\documentclass[generic, preprint]{imsart}
\pdfoutput=1
\RequirePackage[OT1]{fontenc}
\RequirePackage{amsthm,amsmath}
\RequirePackage[numbers]{natbib}
\RequirePackage[colorlinks,citecolor=blue,urlcolor=blue]{hyperref}

\startlocaldefs
\renewcommand{\cite}{\citet}
\numberwithin{equation}{section}
\theoremstyle{plain}

\usepackage{jan-rudi-abrev-package}
\newtheoremstyle{mysc}
  {3pt}
  {3pt}
  {\it}      
  {}
  {\sc}
  {.}
  {.5em}
  {}
\newtheoremstyle{myex}
  {10pt}
  {10pt}
  {\rm}
  {}
  {\sc}
  {.}
  {.5em}
  {}
\theoremstyle{mysc}\newtheorem{prop}{Proposition}[section]
\theoremstyle{mysc}\newtheorem{assumption}{Assumption}[section]
\theoremstyle{mysc}\newtheorem{coro}[prop]{Corollary}
\theoremstyle{mysc}\newtheorem{theo}[prop]{Theorem}
\theoremstyle{mysc}
\theoremstyle{mysc}\newtheorem{lem}[prop]{Lemma}
\theoremstyle{myex}\newtheorem{rem}{Remark}[section]
\theoremstyle{myex}
\theoremstyle{myex}
\numberwithin{equation}{section} 

\endlocaldefs

\begin{document}

\begin{frontmatter}
\title{On rate optimal local estimation in functional linear regression
%\thanksref{T1}
}
\runtitle{On rate optimal local estimation in FLR}
%\thankstext{T1}{Footnote to the title with the `thankstext' command.}

\begin{aug}
\author{\fnms{Jan} \snm{Johannes}%\thanksref{t1,t2}
\ead[label=e1]{jan.johannes@uclouvain.be}}
\and
\author{\fnms{Rudolf} \snm{Schenk}%\thanksref{t3}
\ead[label=e2]{rudolf.schenk@uclouvain.be}}

\address{
Institut de statistique, biostatistique et sciences actuarielles (ISBA),\\ Voie du Roman Pays 20, B-1348 Louvain-la-Neuve,
    Belgium.\\
\printead{e1,e2}}

% \author{\fnms{Third} \snm{Author}
% \ead[label=e3]{third@somewhere.com}
% \ead[label=u1,url]{www.foo.com}}
% 
% \address{Address of the Third author\\
% usually few lines long\\
% usually few lines long\\
% \printead{e3}\\
% \printead{u1}}

%\thankstext{t1}{Some comment}
%\thankstext{t2}{First supporter of the project}
%\thankstext{t3}{Second supporter of the project}
\runauthor{J. Johannes and R. Schenk}

\affiliation{Universit\'{e} catholique de Louvain}

\end{aug}

\begin{abstract}
We consider the estimation of the value  of a 
linear functional of the slope parameter in  functional linear regression,
where scalar responses  are modeled in dependence of random functions. 
The theory in this paper  covers 
in particular point-wise estimation as well as the estimation of weighted averages of the slope parameter.
We propose a  plug-in estimator which is based on a dimension reduction technique and additional thresholding. 
It is shown that this estimator is consistent under mild assumptions. 
We derive a lower bound  for the maximal  mean squared error of any estimator over a certain ellipsoid of slope parameters
and a certain class of  covariance operators associated with the regressor.
It is shown that 
 the proposed estimator attains this lower bound up to a constant and hence it  is  minimax optimal.
Our results are
appropriate to discuss a
wide range of possible regressors, slope parameters and functionals. They are illustrated by considering  the point-wise estimation of
 the slope parameter or its derivatives and its average value over a given interval.

\end{abstract}


\tableofcontents
\end{frontmatter}

\section{Introduction}
\label{sec:intro}
 A common problem in functional regression is to investigate the dependence of a real random variable $Y$ on the variation of an explanatory random function $X$. It is usually assumed 
that  the regressor $X$ takes its  values in a separable Hilbert space 
$\mmH$
which is endowed with  an inner product $
\Hskalar
$ and its induced  norm $\Hnorm$. 
For convenience,  the regressor $X$ is often supposed to be centered in the  sense that for all $\func \in \mmH$ the real valued random variable $\HskalarV{X,f}$  
 has mean zero.
In this paper, the   
dependence of $Y$ on $X$ is supposed to be linear, that is
\begin{equation}\label{intro:e1}Y=\HskalarV{\sol,X} +\sigma\epsilon,\quad\sigma>0,
\end{equation}
with an unknown slope parameter $\sol\in
 \mmH
$  and a centered and standardized error term $\epsilon$. 
We focus on the estimation of the value  of a known linear functional of the slope $\sol$, which we denote by $\ell(\sol)
$.
The non-parametric estimation of the value of a linear functional from Gaussian white
noise observations is  subject of considerable literature (in case of direct observations see \cite{Speckman1979}, \cite{Li1982} or \cite{IbragimovHasminskii1984}, while in case
of indirect observations we refer to \cite{DonohoLow1992}, \cite{Donoho1994} or \cite{GoldPere2000} and references therein).
In the  literature, the most studied examples for estimating linear functionals  are point-wise estimation of $\sol$
and the estimation  of (possibly weighted) averages over a subinterval of its domain. These examples are particular cases of our general setting.
The objective of this paper is to establish a minimax theory for the non-parametric estimation of the value of a linear functional of the slope parameter $\sol$ in the functional linear model as
considered  in \eqref{intro:e1}, which    in general does not
lead to Gaussian white noise observations. 
 For this purpose we use a  plug-in estimator $\hell_m:=\ell(\hsol_m) $
based on an estimator $\hsol_m$ of the slope parameter that has been  proposed
by \cite{CardotJohannes2008} and  is inspired  by  the  linear Galerkin approach
coming from the inverse problem community  (c.f.
\cite{EfromovichKoltchinskii2001} or \cite{HoffmannReiss04}). 
 In recent years, the non-parametric estimation of the slope function $\sol$ from an independent and identically distributed (i.i.d.) sample of $(Y,X)$ has been of growing interest in the literature. For example, \cite{Bosq2000},
\cite{CardotMasSarda2007} or \cite{MullerStadtmuller2005} 
analyze a functional principal components regression, while a penalized least squares approach combined with projection
onto some basis (such as splines) is studied in \cite{RamsayDalzell1991}, \cite{EilersMarx1996}, \cite{CardotFerratySarda2003}, \cite{HallHorowitz2007} or
  \cite{CrambesKneipSarda2007}. 
All the proposed estimators of $\sol$  have in common that they achieve
under reasonable
assumptions  only  very poor rates of convergence. In other words, even
relatively large sample sizes may not be much of a help for  estimating the
slope parameter accurately as a whole.
The reason for these poor convergence rates is intrinsic to the considered model
as it leads in a natural way to an ill-posed inverse problem. To be more
precise, as considered for example 
in \cite{Bosq2000}, \cite{CardotFerratySarda2003} or
\cite{CardotMasSarda2007}, we suppose that the regressor $X$ has a finite second
moment, i.e., $\Ex\HnormV{X}^2<\infty$, and that $X$ is uncorrelated to the
random error $\epsilon$ in the sense that $\Ex{[\epsilon\HskalarV{X,f}]}=0$ for
all $f\in\mmH$. Multiplying both sides in (\ref{intro:e1}) by $\HskalarV{X,f}$
and taking the expectation leads to the continuous equivalent of the normal
equation in a
classical multivariate linear model. That is,  we have for all $\func \in \mmH$
\begin{equation}\label{intro:e2:2}\HskalarV{\gf,\func}:=\Ex[Y\HskalarV{X,\func}]
=\Ex[\HskalarV{\sol,X}\HskalarV{X,\func}]=:\HskalarV{\op \sol,\func},
\end{equation}
where $\gf$ belongs to $\mmH$ and $\op$ denotes the covariance operator associated with the random function $X$.
 In what follows we always assume that there exists a unique  solution $\sol\in \mmH$ of equation (\ref{intro:e2:2}), i.e., that $\op$ is strictly positive and that $\gf$ belongs to its range  (for a detailed discussion we refer to \cite{CardotFerratySarda2003}). 
Obviously, these conditions ensure as well that the value of a linear functional of $\sol$ is identified.
 Since the estimation of $\sol$ involves the
inversion of the covariance operator $\op$  it is called an inverse problem. Moreover, due to the finite second moment of the regressor $X$, the associated covariance
operator $\op$ is  nuclear, i.e., its trace is finite. Consequently, unlike in a multivariate  linear model,  a continuous generalized inverse of $\op$ does not exist as long as the range of $\op$ is
an infinite dimensional subspace of $\mmH$.  Therefore, the  reconstruction of
$\sol$ is ill-posed 
(with the additional difficulty that $\op$ is unknown and has to be estimated). 
As usual in the context of ill-posed inverse problems we impose additional conditions on the unknown slope parameter
$\sol$ and the covariance operator $\op$ which will be expressed in the form $\sol\in\cF$ and $\op\in \opclass$, for suitably chosen classes $\cF\subseteq \mmH$ and $\opclass$. 
The class $\cF$  reflects prior information on the solution $\sol$, e.g., its level of smoothness, and will be constructed flexibly enough to characterize, in particular,  differentiable  or analytic functions.
The class $\opclass$ links the mapping properties of the operator $\op$ to the
regularity conditions imposed on the slope function $\sol$. 
Typically, the assumption $\op\in\cG$  results in conditions on the decay of the
eigenvalues of the operator $\op$. The construction of the class $\cG$ allows us to discuss both a polynomial and
exponential decay of those eigenvalues.
It is interesting to note that  \cite{CaiHall2006} also consider the estimation of a linear functional in functional linear regression.
However, their results are restricted to differentiable slope parameters $\sol$ and polynomially decreasing eigenvalues of the operator $\op$. Moreover, the restrictions imposed on the linear  functional $\frep$   by \cite{CaiHall2006} implicitly exclude the particularly interesting case of point-wise estimation.\\
 We shall assess the accuracy of the proposed plug-in estimator $ \hell_m$
of the value $\ell(\sol)$
 by its maximal mean squared error over the classes $\cF$ and $\opclass$, that is,
\begin{equation} \label{eq:def:risk}
 \risk[\hell_m, \cF,\opclass]:=\sup_{\sol\in\cF}\sup_{\op\in\opclass}
\Ex|\widehat \ell_m-\ell(\sol)|^2.
\end{equation}
Below we derive   a lower  bound for  $\risk[\tilde \ell, \cF,\opclass]$
for all  estimators $\tilde\ell$ 
 and show that it provides up to a constant $C>0$ also an upper bound for the maximal risk over $\cF$ and $\opclass$ of the estimator $\hell_m$, i.e.,
\begin{equation*}
\risk[\hell_m, \cF,\opclass]\leq C\cdot\inf_{\tilde\ell}\risk[\tilde\ell, \cF,\opclass],
\end{equation*}  
 where the infimum is taken over all estimators of $\ell(\sol)$.
We thereby prove the minimax optimality of the estimator $\hell_m$. Our results are
appropriate to discuss a
wide range of possible regressors, slope parameters and functionals.
Moreover they  yield in a natural way 
uniform bounds if the functional  varies over a certain subset of the dual space.
The paper is organized in the following way:
in Section 
\ref{sec:mandn}
 we 
develop the plug-in estimator
 and
introduce our basic assumptions. In particular we define and illustrate the classes $\cF$ and $\cG$. Moreover, we embed the cases of point-wise and local average estimation in this general framework and present the resulting plug-in estimators. In this situation the classes  $\cF$ and $\cG$ cover the often considered cases of Sobolev
 ellipsoids  
and finitely or infinitely 
smoothing  covariance operators.
We provide, in Section \ref{sec:tmp},   sufficient conditions for  consistency of the proposed plug-in estimator 
 and then show its minimax-optimality.
More precisely, we derive a lower bound for the maximal risk over the classes $\cF$ and $\cG$  based on
 an i.i.d.\ sample obeying the functional linear model \eqref{intro:e1}. 
We show that the proposed plug-in estimator 
 attains the lower bound  up to a constant for a wide range of classes $\cF$ and $\cG$. These results are used to discuss the point-wise estimation of
 the slope parameter or its derivatives  and its average value over a given interval. 
The proofs can be found in the appendix.
\section{Methodology and notations}\label{sec:mandn}
\subsection{Thresholding projection estimator}\label{sec:tpe}
Following  \cite{CardotJohannes2008} we construct an estimator of the  unknown
slope function $\sol$ using a linear Galerkin approach. The estimation of $\sol$ is based on  a dimension reduction together with an additional thresholding, which we elaborate in the following. Let us  specify 
an arbitrary orthonormal basis $\{\bas_j\}_{j=1}^\infty$ of $\mmH$.
 We require in the following that the slope function $\sol$ belongs to a function class $\cF$ containing $\{\bas_j\}_{j= 1}^{\infty}$ and, moreover that $\cF$
is included in the domain  of the linear functional $\ell$. For technical reasons and without loss of generality we assume  that $\ell(\bas_1)\neq0$ which can always be ensured by reordering, except for the trivial case $\ell\equiv 0$.
With respect to this basis we consider
for $f\in\mmH$ the expansion $f=\sum_{j=1}^\infty\fouv{f}_j\bas_j$, with $
\fouv{f}_j:=\HskalarV{f,\bas_j}$, for $j\geq 1$. The unknown solution
$\sol\in \mmH$ is hence uniquely determined by its coefficients
$(\fouv{\sol}_j)_{j\geq 1}$.
Given an  integer dimension parameter $m\geq 1$, we consider the subspace $\sspace_m$ 
spanned by the  functions $\{\bas_j\}_{j=1}^m$. We  recall that a Galerkin   solution $\sol_{m}\in\sspace_m$ of the operator equation \eqref{intro:e2:2} with respect to $\sspace_m$ satisfies
\begin{equation}\label{app:unknown:Galerkin:1:2}
\HnormV{\gf-\op\sol_m}\leqslant  \HnormV{\gf-\op\tilde\sol},\quad\forall \tilde \sol\in\sspace_m.
\end{equation}
Since $\op$ is strictly positive, the Galerkin solution exists in a  unique way.
Precisely,
 if we consider  the expansion $\sol_{m}=\sum_{j=1}^m \fouv{\sol_m}_j\bas_j$
 then $\sol_{m}$ is uniquely determined by the vector of coefficients $\fouv{\sol_{m}}_{\um}:= (\fouv{\sol_m}_1,\ldots,\fouv{\sol_m}_m)^t$. In this context, the restriction of $\op$ to an operator from $\sspace_m$ to itself can be identified with a matrix operating on 
$\R^m$.
This matrix is given by the entries $\HskalarV{\bas_j, \op\bas_l}$ for $1\leq j,l\leq m$ and will be denoted as $\foum{\op}_{\um}$, in slight abuse of notation. It is easy to verify that the Galerkin solution defined by \eqref{app:unknown:Galerkin:1:2} satisfies  $\foum{\op}_{\um}\fouv{\sol_m}_{\um}=\fouv{\gf}_{\um}$. Since $\op$ is strictly positive,
the matrix $\foum{\op}_{\um}$ is nonsingular for all $m\geq 1$, such that its inverse $\foum{\op}_{\um}^{-1}$ always exists. Therefore, the Galerkin solution $\sol_{m}$ is determined by 
\begin{equation}\label{eq:gal:th}\fouv{\sol_{m}}_{\um}= \foum{\op}_{\um}^{-1}\fouv{\gf}_{\um}. \end{equation}
 Here and subsequently, we denote by $\{(Y_i, X_i)\}_{i=1}^n$  an  
i.i.d.\ sample of $(Y,X)$ of size $n$ satisfying
\eqref{intro:e1}.  We observe that  $\fouv{\gf}_{\um}=\Ex
Y\fouv{X}_{\um}$ and $\foum{\op}_{\um}=\Ex \fouv{X}_{\um}\fouv{X}_{\um}^t$ and 
hence it 
is natural to consider the estimators
\begin{equation*}\label{com:e2}
\fouv{{\hgf}}_{\um}:=\frac{1}{n}\sum_{i=1}^n Y_{i}\,\fouv{X_i}_{\um}\quad\mbox{and}\quad \foum{\hop}_{\um}:=\frac{1}{n}\sum_{i=1}^n \fouv{X_{i}}_{\um}\fouv{X_i}_{\um}^t\end{equation*}
of 
 $\fouv{{\gf}}_{\um}$ and $\foum{\op}_{\um}$, respectively. The estimator of
$\fouv{\sol_m}_{\um}$ is derived from \eqref{eq:gal:th} by replacing  
$\fouv{{\gf}}_{\um}$ and $\foum{\op}_{\um}$ by their empirical counterparts.
However, the inversion of the empirical covariance matrix $\foum{\hop}_{\um}$
introduces an instability to the estimation procedure even if the matrix
$\foum{\op}_{\um}$ is well-conditioned. This instability issue  is treated  by
an additional thresholding step.
Let us denote by $ \mnormV{\foum{\hop}^{-1}_{\um}}$  the spectral norm of the matrix   $\foum{\hop}_{\um}^{-1}$,  
which equals its largest eigenvalue. Then, the estimator $\hsol_{m}^{\thresh}\in\sspace_m$ of $\sol$ is
determined by the vector of coefficients 
 \begin{equation}\label{eq:def:Galerkin:loesung}
  \fouv{\hsol_m^{\thresh}}_{\um}:=
 \left\{\begin{array}{lcl} 
 \foum{\hop}_{\um}^{-1} \fouv{\hgf}_{\um}, && \mbox{if $\foum{\hop}_{\um}$ is non-singular and }\mnormV{\foum{\hop}^{-1}_{\um}}\leq \thresh\ n,\\
 0,&&\mbox{otherwise}.
 \end{array}\right.
 \end{equation}
In order to estimate the value of the linear functional
$\ell(\sol)$ we consider the plug-in estimator 
 $\hell_{m}:=\ell(\hsol_{m}^{\thresh})$ and observe that $\ell(\hsol_m)
=(\ell(\bas_1),\dotsc,\ell(\bas_m))
 [{\hsol_m}]_{\um}=:[\frep]_{\um}^t[{\hsol_m}]_{\um}$ with the slight abuse of notations  
$[\frep]_{\um}:=([\frep]_j)_{1\leq j\leq m}$  and generic elements $[\frep]_j:=\frep(\bas_j)$. 
The estimator  obviously satisfies
\begin{equation}\label{gen:def:est}
 \widehat{\ell}_m=
 \left\{\begin{array}{lcl} 
 [\frep]_{\um}^t[\hop]_{\um}^{-1} [\hgf]_{\um}, && \mbox{if $[\hop]_{\um}$ is non-singular and }\mnormV{[\hop]^{-1}_{\um}}\leq \thresh\ n,\\
 0,&&\mbox{otherwise}.
 \end{array}\right.
 \end{equation}
This procedure raises the question how to choose the dimension parameter $m$, which depends on the sample size $n$.  It needs to tend to infinity as $n$ increases and we will discuss its optimal choice in Section \ref{sec:lower}.
\subsection{Basic model assumptions}\label{sec:bma}
Let us 
introduce the
 class $\cF$ which we determine by means of a weighted norm in $\mmH$.
Given the orthonormal basis $\{\bas_j\}_{j=1}^\infty$ and  a strictly positive sequence of  weights $(\weight_j)_{j\geq1}$, or $\weight$ for short,  we define for  $f\in\mmH$ the
weighted norm $\wnormV{f}{\weight}^2:=\sum_{j\geq 1}\weight_j\fouv{f}_j^2$.
Furthermore, we define  $\cF_\weight$ as the completion of $\mmH$ with 
respect to $\wnormV{\cdot}{\weight}$. Obviously,
for a non-decreasing sequence $\weight$  the class $\cF_\weight$ is a subspace of $\mmH$. 
In the illustrations of Section \ref{sec:illu} the order of the sequence $\bw$  directly reflects smoothness assumptions on the solution. If there exist an integer $p>0$
and  a constant $c>0$ such that $c^{-1} j^{2p}\leq \bw_j\leq c^{-1} j^{2p}$, or  $\bw_j\asymp j^{2p} $ for short, then this polynomial increase will corresponds to $p$-times differentiable functions. However,  the theory in this paper is not restricted to  polynomially increasing sequences $\bw$.  We also consider an exponential increase, i.e.,  $\bw_j\asymp\exp(j^{2p})$, which is known to specify analytic functions.
We will assume in the following, that there exist
a non-decreasing, unbounded sequence of  weights  $\bw$  with $\bw_1=1$ and a constant $\br>0$ such that the 
 solution $\sol$ belongs to the ellipsoid
 $\cF_{\bw}^\br := \big\{f\in \cF_\bw: \wnormV{f}{\bw}^2\leq \br\big\}$.
In order to guarantee that  $\cF_{\bw}^\br$ is contained in the domain of the linear functional $\frep$ and that $\frep(f)=\sum_{j\geq1} [\frep]_j{[f]}_j$ for all
$f \in\cF_{\bw}^\br$ with
$[\frep]_j=\frep(\bas_j)$, $j\geq1$,  it is sufficient that  $\sum_{j\geq1} [\frep]_j^2\bw_j^{-1}<\infty$. In what follows, we understand arithmetic operations on a sequence of real numbers $\weight$ component-wise, e.g., we write $1/\weight$ instead of $(1/\weight_j)_{j\geq 1}$.
As no confusion can be caused we define $ \wnormV{\frep}{1/\bw}^2 =\sum_{j\geq1} [\frep]_j^2\bw_j^{-1}$ and denote the set of all linear functions with $ \wnormV{\frep}{1/\bw}^2<\infty$ by  $\cFF_{1/\weight}$. 
We may emphasize
that we  neither impose that the sequence $[\frep]=([\frep]_j)_{j\geq1}$ tends to zero nor that it is square summable. However, if it is square summable then the entire of $\mmH$ is the domain of $\frep$. Moreover,
$[\frep]$ coincides with the sequence of generalized Fourier coefficients of the representer of $\frep$ given by Riesz's theorem.
 The assumption $\frep\in\cFF_{1/\bw}$ enables us in specific cases to deal with more demanding functionals, such as the estimation of the point-wise evaluation of the slope. 
As a byproduct, our theory allows us  to assess the performance of the estimation procedure of  $\ell(\sol)$ not only for a single $\frep\in\cFF_{1/\bw}$, 
but also  for  $\frep$ varying over an ellipsoid in $\cFF_{1/\bw}$. For this purpose we 
 suppose that there exists 
a non-negative sequence  $\hw$ with $\hw_1=1$ and
a constant $\hr>0$   such that $\frep$ belongs to the ellipsoid
  $\cFF_\hw^\hr:=\{\ell \in \cFF_{1/\bw}: \wnormV{\ell}{\hw}^2:= \sum_{j\geq1}\hw_j [\frep]_j^2\leq \hr  \}
$. Under the condition $\sup_{j\geq 1}\{1/(\hw_j\bw_j)\}<\infty$ the ellipsoid  $\cFF_\hw^{\hr}$ is clearly a subset of $\cFF_{1/\bw}$. In order to describe the mapping properties   of the covariance operator $\op$, stated in the form $\op\in\opclass$,  we introduce the set $\cN$ of all strictly positive  nuclear operators defined on $\mmH$.
We suppose that there exists  a constant $\Td\geqslant 1$ and a strictly positive, non-increasing sequence of weights $\Tw$ with $\Tw_1=1$ and $\sum_{j=1}^\infty\Tw_j<\infty$,  such that $\op$ belongs to  the class
\begin{equation*}\label{bm:link}
\cN_{\Tw}^{\Td}:=\Bigl\{ T\in\cN:\quad   d^{-2}\wnormV{f}{\Tw^2}^2\leqslant \HnormV{T f}^2\leqslant {\Td^2}\, \wnormV{f}{\Tw^2}^2,\quad \forall f \in \mmH\Bigr\}.
\end{equation*}
Note that for each  $T\in\cN$ the trace $\tr(T):=\sum_{j=1}^\infty \HskalarV{\bas_{j},T\bas_{j}}$ is finite. Hence, setting $\foum{T}_{j,j}:= \HskalarV{\bas_{j},T\bas_{j}}$, $j\geq1$,  the 
sequence  $(\foum{T}_{j,j})_{j\geqslant1}$ converges to zero. Moreover, for $T\in\cN_{\Tw}^\Td$ the decay of this sequence is characterized by $\Tw$  since   
 $d^{-1}\Tw_j\leq \foum{T}_{j,j}\leq d\Tw_j$ for all $j\geq 1$. Furthermore,
if $\lambda$ denotes its  sequence of eigenvalues then  $d^{-1} \Tw_j\leq \lambda_j\leq d \Tw_j$ holds true for all $j\geq1$.
 Let us summarize the conditions on 
the sequences $\bw$, $\hw$ and $\Tw$.
\begin{assumption}\label{ass:reg}
Let  $\bw$, $\hw$ and  $\Tw$ be strictly positive sequences of weights  such that   $\bw$ and $1/\Tw$  are
 non-decreasing. We suppose that the sequences satisfy $\bw_1= \hw_1=\Tw_1= 1$,  $\sup_{j\geq 1}\{1/(\hw_j\bw_j)\}<\infty$,  $\sum_{j=1}^\infty\Tw_j<\infty$ and that $\bw$ tends to infinity.
\end{assumption}
\newpage
\noindent{We illustrate the last assumption for typical choices of the
sequences $\bw$,  $\hw$ and $\Tw$:}
\begin{enumerate}
 \item[\textit{(ppp)}]
Consider $\bw_j\asymp|j|^{2p}$,   $\Tw_j\asymp|j|^{-2a}$
 and either
 \textit{(i)} $[\frep]^2_j\asymp |j|^{-2s}$,  or
\textit{(ii)}
$\hw_j\asymp |j|^{2s}$ then Assumption \ref{ass:reg} holds true if $p>0$, $a>1/2$,  and either \textit{(i)} $s>1/2-p$ or \textit{(ii)} $s>-p$.
\item[\textit{(pep)}] Consider
$\bw_j\asymp|j|^{2p}$, $\Tw_j  \asymp \exp(-|j|^{2a})$
 and either 
\textit{(i)} 
 $[\frep]_j^2\asymp |j|^{-2s}$ or
\textit{(ii)}
$\hw_j\asymp |j|^{2s}$ then Assumption \ref{ass:reg} holds true if  $p>0$, $a>0$ and   either 
\textit{(i)}  $s>1/2-p$ or \textit{(ii)}  $s>-p$. 
\item[\textit{(epp)}] 
Consider
$\bw_j\asymp \exp(|j|^{2p})$,
$\Tw_j\asymp|j|^{-2a}$ 
 and either
\textit{(i)} $[\frep]^2_j\asymp|j|^{-2s}$ or
\textit{(ii)} $\hw_j^2\asymp|j|^{2s}$
then Assumption \ref{ass:reg} holds true if $p>0$, $a>1/2$
and  $s\in\R$.
\item[\textit{(ppe)}] 
 Consider $\bw_j\asymp |j|^{2p}$, $\Tw_j\asymp |j|^{-2a}$
 and either
\textit{(i)} 
$[\frep]_j^2\asymp\exp(-|j|^{2s})$  or
\textit{(ii)}  $\hw_j\asymp\exp(|j|^{2s})$ then  
 Assumption \ref{ass:reg} holds true if  $p>0$, $a>1/2$
and  $s>0$.
\end{enumerate}
\begin{rem}
\cite{CaiHall2006}
  consider only  \textit{(i)} in the case \textit{(ppp)} 
 and suppose  a decay of the representing coefficients $ [\frep]$  of order
$(|j|^{-s})$ with $s>1/2$. 
 This condition  excludes, for example, point-wise
estimation which we will consider, together with the other four cases, below.
\hfill$\square$ \end{rem}
The only assumptions on the stochastic behavior of the error term $\epsilon$ and the regressor $X$  that we need in order to derive our mean squared error results concern their moments.  
We observe that for all $f\in\mmH$ the random variable $\HskalarV{f,X}$ has
mean zero and variance $\HskalarV{\op\func,\func}$ and we will impose moment
conditions on the standardized random variable $
\HskalarV{\op f,f}^{-1/2}\HskalarV{f,X}$.
 \begin{assumption}\label{ass:MOM} 
 There exist an  integer $k\geq 12 $ and a constant $\eta\geq 1$ such that  $\Ex\absV{\epsilon }^{4k}\leq \eta$
and that for all $f\in \mmH$ with $\HskalarV{\op f,f}=1$  it holds  $\Ex\absV{\HskalarV{f,X} }^{4k}\leq \eta$.
\end{assumption}
Note that any centered Gaussian random function $X$  with finite second moment satisfies Assumption \ref{ass:MOM}, since
for all $f\in\mmH$  with $\HskalarV{\op f,f}=1$  the corresponding random variable
$\HskalarV{f,X}$
 is 
standard normally distributed and consequently
  $\Ex\absV{\HskalarV{f,X}^{4k}}\leq(4k-1)\cdot(4k-3)\cdot\ldots\cdot 5\cdot 3\cdot 1$. 
\subsection{Point-wise and   local average estimation}\label{subs:pwlae}
Consider $\mmH=L^2[0,1]$ with its usual norm and inner product and  the trigonometric basis
\begin{equation*}\label{bm:def:trigon:1}
\psi_{1}:\equiv1, \;\psi_{2j}(s):=\sqrt{2}\cos(2\pi j s),\; \psi_{2j+1}(s):=\sqrt{2}\sin(2\pi j s),s\in[0,1],\; j\in\N.\end{equation*} 
Recall the typical choices of the sequences $ \bw,\hw$  and $\Tw$ as introduced above.
If  $\weight_j\asymp |j|^{2p}$ for a positive integer $p$, see  cases \textit{(ppp)},\textit{(pep)},\textit{(ppe)}, then  
 the subset $\cF_{\weight}$ coincides with the Sobolev space  of $p$-times differential periodic functions
(c.f. \cite{Neubauer1988,Neubauer88}). 
In the case \textit{(epp)}
  it is well-known that for 
$p>1$ every 
$f\in \cF_{\weight}$ is 
an analytic function (c.f. \cite{Kawata1972}).
On the other hand we consider two special cases describing a \lq\lq regular decay\rq\rq\ 
of the unknown eigenvalues of $\op$. Precisely,  we assume  a polynomial decay of $\Tw$ with $a>1/2$ 
in
the cases \textit{(ppp), (epp)} and \textit{(ppe)}.
Easy calculus 
shows that the covariance operator $\op\in\cN_\Tw^\Td$ acts for integer $a$  
    like integrating   $(2a)$-times and hence it is called 
{\it finitely smoothing} (c.f. \cite{Natterer84}). 
In the  case \textit{(pep)} we assume  an exponential decay of $\Tw$ and
it is 
 easily  seen that  the range of  $\op\in \cN_{\Tw}^{\Td}$ is a subset of $C^\infty[0,1]$, 
therefore  the operator
is called {\it infinitely smoothing} (c.f. \cite{Mair94}). 
\paragraph{Point-wise estimation}
By  \textit{evaluation   in a given point}
 $t_0 \in [0,1]$ we mean the linear functional $\frep_{t_0}$ mapping  $f$ to $f(t_0):=\ell_{t_0}(f)=\sum_{j=1}^\infty [f]_j\bas_j(t_0)$.
 In the following we shall assume that the point evaluation is well-defined on
 the set of slope parameters $\cF_{\bw}$ which is obviously implied by  $\sum_{j=1}^\infty[\ell_{t_0}]_j^2\bw_j^{-1}<\infty$. 
 Consequently, the condition  $\sum_{j\geq1}\bw_j^{-1}<\infty$  is sufficient  to guarantee that the point evaluation is well-defined on $\cF_\bw$.
 Obviously, in case  \textit{(pep)} or in other words  for exponentially increasing $\bw$,  this additional condition is automatically satisfied. However, a polynomial increase, as in the cases  \textit{(ppp)} and \textit{(ppe)},  
 requires the  assumption $p>1/2$. Roughly speaking, this means that the slope parameter has at least to be  continuous. In order to estimate the value
 $\sol (t_0)$ we consider the plug-in estimator 
 \begin{equation*}
 \hell_{t_0}^m
 =
 \left\{\begin{array}{lcl} 
[\ell_{t_0}]_{\um}^t[\hop]_{\um}^{-1} [\widehat{g}]_{\um}, && \mbox{if $[\hop]_{\um}$ is non-singular and }\mnormV{[\hop]^{-1}_{\um}}\leq  n,\\
 0,&&\mbox{otherwise},
 \end{array}\right.
 \end{equation*}
 with $[\ell_{t_0}]_{\um}=(\psi_1(t_0),\dotsc,\psi_m(t_0))^t$. Moreover, we observe that $\hell_{t_0}^m=\ell_{t_0}(\hsol_m)=\hsol_m(t_0)$
 for $\hsol_m^{\thresh}\in\sspace_{m}$ 
  as determined by  \eqref{eq:def:Galerkin:loesung}.

\paragraph{Local average estimation} 
 Next we are  interested in the average value   of $\sol $ on the interval $[0,b]$ for $b\in(0,1]$. If we denote the linear functional mapping $f$ to 
$b^{-1}\int_{0}^b f (t)dt$ by $\ell^b$, then it is easily seen that 
$[\ell^b]_1=1$, $[\ell^b]_{2j}=(\sqrt2\pi jb)^{-1}\sin(2\pi jb)$,
 $[\ell^b]_{2j+1}=(\sqrt2\pi j b)^{-1}\cos(2\pi jb)$  for $j\geq1$.  In this situation  the plug-in estimator $ \hell^b_{m}= b^{-1}\int_0^b\hsol_m(t)dt$ is written as
\begin{equation*}
\hell^{b}_m
=
\left\{\begin{array}{lcl} 
[\ell^b]_{\um}^t[{\hop}]_{\um}^{-1} [\widehat{g}]_{\um}, && \mbox{if $[{\hop}]_{\um}$ is non-singular and }\mnormV{[{\hop}]^{-1}_{\um}}\leq  n,\\
0,&&\mbox{otherwise}.
\end{array}\right.
\end{equation*}
Again, with  $\hsol_m^{\thresh}\in\sspace_{m}$ as
 determined by  \eqref{eq:def:Galerkin:loesung}, we observe that $ \hell^{b}_m= {\ell(\hsol_m)}=b^{-1}\int_0^b\hsol_m(t)dt$.\par
\newpage\section{Theoretical properties}\label{sec:tmp}
%\setcounter{chapter}{2}
%\setcounter{equation}{0} %-1
%\setcounter{section}{0}
%\setcounter{prop}{0}
%\setcounter{rem}{0}
%\noindent {\bf \arabic{chapter}. Theoretical main part.}
%\section{Optimal local estimation  in the general case.}
\subsection{Consistency under mild assumptions}\label{sec:consis}
In this section we provide sufficient conditions for the consistency of the
estimator ${\widehat\ell}_m$  defined in \eqref{gen:def:est} for all $\sol
\in \cF_{\bw}$ and $\frep \in \cFF_{1/\bw}$. Recall that this estimator is based
on the Galerkin solution
$\sol_{m}$.
The next assertion summarizes that 
 consistency can be ensured if
\begin{equation}
 \label{eq:suff:cond:Galerkin}
\wnormV{\sol-\sol_{m}}{\bw}=o(1)\quad \text{ as $m\to\infty$}
\end{equation}
  which is in general not
satisfied   without further assumptions, not even for $\bw \equiv 1$. 
\begin{prop}\label{res:gen:prop:cons} 
Let $\{(Y_i, X_i)\}_{i=1}^n$ be an i.i.d.\ sample of $(Y,X)$ satisfying
\eqref{intro:e1}.
Suppose that $\sol\in\cF_\bw$ and $\frep \in \cFF_{1/\bw}$, where the sequence
$\bw$  satisfies     Assumption \ref{ass:reg} and let 
 Assumption \ref{ass:MOM} hold true.
Consider the estimator ${\widehat\ell_m}$ 
with  dimension  $m=m(n)$ satisfying $1/m=o(1)$ and $m^3=O(n)$ as
$n\to\infty$
 and suppose 
that  in addition the following conditions are fulfilled
\begin{equation}\label{eq:bed:red}
                \mnormV{\foum{\op}^{-1}_{\um}}=o(n)\quad\text{ and
}\quad \fouv{\frep}_{\um}^t\foum{\op}_{\um}^{-1}\fouv{\frep}_{\um}
=o(n)\quad\text{
as $n\to\infty$}.
                \end{equation}
If condition \eqref{eq:suff:cond:Galerkin} holds true 
then  we have
 $\Ex|\widehat\ell_m-{\ell(\sol)}|^2=o(1)$ as  $n\to\infty$. 
  \end{prop}
If the operator $\op$ satisfies a link condition, i.e., $\op\in\cN_\Tw^\Td$,
then 
condition \eqref{eq:suff:cond:Galerkin} is 
automatically fulfilled which is expressed in the next assertion.
 \begin{coro}\label{res:gen:coro:cons}
Let the covariance operator  $\op$   be an element of  $\cN_\Tw^\Td$  with $\Td\geqslant1$ and let the sequence $\Tw$  satisfy Assumption  \ref{ass:reg}. 
 The conclusion of Proposition \ref{res:gen:prop:cons} still holds true without imposing condition \eqref{eq:suff:cond:Galerkin} provided \eqref{eq:bed:red} is substituted by 
\begin{equation}\label{eq:neu:3.3}
\frac{1}{n\Tw_m} =o(1),\quad\text{and}\quad \frac{1}{n}\sum_{j=1}^m \frac{\fouv{\frep}^2_j}{\Tw_j}=o(1)
\quad\text{ as  $n\to\infty$.} 
\end{equation}
\end{coro}
\begin{rem} Consider the case \textit{(ppp)} and suppose  that  $\sum_{j\geq1} [\frep]^2<\infty $. In this situation,
the assumption $m^{2a}=o(n)$ as $n\to\infty$ implies the additional condition
in Corollary \ref{res:gen:coro:cons}. Interestingly, in a
direct regression  model the condition
$m=o(n)$ as $n\to\infty$ is needed to ensure consistency, which would correspond to the case $a=1/2$. This, however,  cannot hold true for any $X$ with finite second moment. 
\null\hfill $\square$
\end{rem}
\subsection{The lower bound}\label{sec:lower}
In order to  obtain a lower bound for the minimax risk
$\inf_{\tilde\ell}\risk[\tilde\ell,\cF_\bw^\br,\cN_\Tw^\Td]$ defined in \eqref{eq:def:risk}   we assume in  addition
 that for all $\func \in \mmH$ the conditional distribution of
$\epsilon$
given $\HskalarV{X,\func}$ is Gaussian with mean zero and variance one, or
$\epsilon|\HskalarV{X,f}\sim \cN(0,1)$ for short. This assumption   is only used to simplify the calculation of
the distance  between distributions of the observations corresponding to different slope functions.
In order to formulate the lower bounds below let us define 
\begin{equation}
\label{eq:def:kstar:astar} 
\kstar:= \argmax_{m \geq1}\left\{\frac{\min\{\frac{\Tw_m}{\bw_m},n^{-1}\}}{
\max\{\frac{\Tw_m}{\bw_m},n^{-1}\}}\right\}, \quad
\astar:=\max\left\{ \frac{\Tw_\kstar}{\bw_\kstar},n^{-1}\right\}\quad\text{for all $n\geq 1$.}
\end{equation}
The lower bound needs the following assumption.
\begin{assumption}
 \label{ass:Kappa}
Let $\bw$ and $\Tw$ be sequences such that 
\begin{equation}\label{eq:def:kappa}
 0<\kappa:=\kappa(\Tw,\bw):=\inf_{n\geq 1}
\Bigl\{(\astar)^{-1}\min\Bigl\{\frac{\Tw_\kstar}{\bw_\kstar}, n^{-1}\Bigl\}\Bigl\}\, \leq 1.
\end{equation}
\end{assumption}
\noindent Now we are in the position to state the main result of this section.
\begin{theo}\label{res:lower}
Let $\{(Y_i, X_i)\}_{i=1}^n$ be an i.i.d.\ sample of $(Y,X)$ satisfying \eqref{intro:e1}.
Suppose that 
Assumptions \ref{ass:reg}, \ref{ass:MOM} and
 \ref{ass:Kappa} hold true
and assume
in addition that
$\epsilon|\HskalarV{X,f}\sim \cN(0,1)$ 
for all $f\in \mmH$.
 If $\kstar\in\N$ is  given by equation
\eqref{eq:def:kstar:astar} then 
we
have 
for all $\frep\in\cFF_{1/\bw}$
and all $n\geq 1$:  
\begin{equation*}  
\inf_{\tilde\ell}\risk[\tilde\ell,\cF_\bw^\br,\cN_\Tw^\Td]
\geqslant  
\frac{\kappa}{4}\,\min \left( \frac{\sigma^2}{2\Td},{\br}\right)
\max\left\{\astar\sum_{j=1}^\kstar
\frac{\fouv{\frep}_j^2}{\Tw_j},\sum_{j>\kstar}\frac{\fouv{\frep}_j^2}{\bw_j}
\right\} .
\end{equation*}
\end{theo}
\begin{rem}
Below we derive an upper 
bound for $\risk[\widehat\ell_{\kstar}, \cF_\bw^{\br},\cN_{\Tw}^{\Td} ]$  of the estimator $\widehat\ell_{\kstar}$  assuming that 
the error term $\epsilon$ and the regressor $X$  are uncorrelated. Obviously, in
this situation Theorem \ref{res:lower} provides a lower bound for any estimator
as long as 
Assumption  \ref{ass:MOM}
 does not exclude
a Gaussian error.  
It is worth to note that  
the lower bound tends to zero with  parametric rate $n^{-1}$   if  and only if  
$\sum_{j=1}^\infty [\frep]^2_j{\Tw_j}^{-1}<\infty$, independently of the
class $ 
\cF_{\bw}^{\br}
$ of slope parameters.\hfill$\square$\end{rem} A straightforward consequence of Theorem \ref{res:lower} is the following lower
bound over the class $\cFF_\hw^\hr$ of functionals. 
We define $j_*:={\argmax}_{1\leq j\leq \kstar}({{\hw_j^{-1}\Tw_j^{-1}}})$ and  the functional $\ell_*$ given by
$[\frep_*]_j:= ( {\hr}/{\hw_{j_*}})^{1/2}$ for $j=j_*$ and 0 otherwise.  Obviously $\frep_*$ is an element of $\cFF_\hw^\hr$. By evaluating the lower bound given by Theorem \ref{res:lower}  for the specific functional $\frep_*$ we immediately obtain 
the following result and we omit its proof.
\begin{coro}\label{coro:lower:alt}
Under the conditions of Theorem \ref{res:lower} we have
\begin{equation*}  
\inf_{\tilde\ell}\sup_{\frep \in \cFF_\hw^{\hr}}
\risk[\tilde\ell,\cF_\bw^\br,\cN_\Tw^\Td]
\geqslant  
\frac{\kappa\tau}{4}\,\min \left( \frac{\sigma^2}{2\Td},{{\br}}\right) \,
\max_{1\leq j\leq \kstar}\Bigl(\frac{1}{\Tw_j\hw_j}\Bigr)\astar
.
\end{equation*}
\end{coro}
\begin{rem}
It is easily seen that 
 the lower bound given in Corollary \ref{coro:lower:alt} tends to zero if and
only if $(\hw_j\bw_j)_{j\geq1}$ tends to infinity. 
In other words,  consistency of an estimator of $\ell(\sol)$ uniformly over
 spheres in  
$\cF_{\bw}$  and $ \cFF_{1/\bw}$
is impossible.
This obviously reflects the ill-posedness of the underlying inverse
problem.\hfill$\square$
\end{rem}
By considering the typical choices of $\bw$ and $\Tw$ gathered in the
cases \textit{(ppp)},\textit{(pep)}, \textit{(ppe)} and \textit{(epp)} above,
we illustrate now the lower bounds 
for the minimax risks
$\inf_{\tilde\ell}\risk[\tilde\ell,\cF_\bw^\br,\cN_\Tw^\Td] $ and
$\inf_{\tilde\ell}\sup_{\frep\in\cFF_{\hw}^\hr}\risk[\tilde\ell,\cF_\bw^\br,
\cN_\Tw^\Td ] $.
We see 
from  Theorem~\ref{res:lower} and Corollary~\ref{coro:lower:alt},
respectively, that their orders  are determined by the sequences
$\delta^*:=(\dstar)_{n\geq 1}$ and $\Delta^*:=(\Dstar)_{n\geq
1}$ given by
\begin{equation}
 \label{eq:def:dstar}
\dstar:=\max\left\{\astar\sum_{j=1}^\kstar
\frac{\fouv{\frep}_j^2}{\Tw_j},\sum_{j>\kstar}\frac{\fouv{\frep}_j^2}{\bw_j}
\right\}
\quad\text{and}\quad 
\Dstar:= \max_{1\leq j\leq \kstar}\Bigl(\frac{1}{\Tw_j\hw_j}\Bigr)\astar.
\end{equation}
In the
next assertion
 we present the orders of those sequences. 
\begin{prop}\label{lem:applic:1}
Let the assumptions of Theorem \ref{res:lower} hold true. Under the following conditions the Assumptions \ref{ass:reg} and \ref{ass:Kappa}
are satisfied and the lower bounds are determined by the orders of $\delta^*$ and $\Delta^*$ as given below.
\begin{enumerate}
 \item[(ppp)]
 If  $p>0$ and $a>1/2$, 
 then $\kstar\asymp n^{1/(2p+2a)}$
 and if
\begin{enumerate}
 \item[(i)]  $s>1/2-p$, then\\
 $\dstar\sim
\begin{cases}
 n^{-(2p+2s-1)/(2p+2a)}, &\text{if $s-a<1/2$}\\
 n^{-1}\log(n), &\text{if $s-a=1/2$}\\
 n^{-1}, &\text{if $s-a>1/2$},
\end{cases}
$ 
\item[(ii)]
 $s>-p$, then 
 $\Dstar\sim
\max(n^{-(p+s)/(p+a)} ,n^{-1}).
$
\end{enumerate}
\item[(pep)] If
 $p>0$ and $a>0$,
 then  $\kstar\asymp \log(n[\log(n)]^{-p/a})^{1/(2a)}$
 and if 
\begin{enumerate}
 \item[(i)] 
 $s>1/2-p$, then 
$\dstar\sim
[\log(n)]^{-(2p+2s-1)/(2a)}$,
\item[(ii)]
 $s>-p$, then
 $\Dstar\sim
[\log(n)]^{-(p+s)/a}.
$
\end{enumerate}
\item[(epp)] 
If 
$p>0$,
 $a>1/2$ and $s\in\R$ then $\kstar\asymp \log(n[\log(n)]^{-a/p})^{1/(2p)}$ and
\begin{enumerate}
 \item[(i)]   
 $ \dstar\sim
\begin{cases}
 n^{-1}[\log(n)]^{(2a-2s+1)/(2p)}, &\text{if $s-a<1/2$}\\
 n^{-1}\log[\log(n)], &\text{if $s-a=1/2$}\\
 n^{-1}, &\text{if $s-a>1/2$},
\end{cases}
$
\item[(ii)]    
$\Dstar\sim 
\max(n^{-1}[\log(n)]^{(a-s)/p} ,n^{-1}).$
\end{enumerate}
\item[(ppe)] 
 If  $p>0$, $a>1/2$ and $s>0$ then $\kstar\asymp n^{1/(2p+2a)}$  and
\begin{enumerate}
 \item[(i)] 
$\dstar\sim n^{-1}
$
\quad(ii) 
$\Dstar\sim 
n^{-1}.$
\end{enumerate}
\end{enumerate}
\end{prop}
\begin{rem}\label{rem:fin:upper:com:1}The rates given in Proposition \ref{lem:applic:1} determine up to a constant the minimax optimal rate of convergence, as 
we will show in Proposition \ref{lem:applic:2} below.
Nevertheless, we shall already emphasize here the interesting influence of the parameters $p$, $s$ and $a$ 
characterizing the \lq smoothness\rq\  of $\sol$, $\frep$ and the decay  of the eigenvalues of $\op$ respectively.
 As we see from Proposition \ref{lem:applic:1},   an increase of  the value of $a$ leads in each case to a slower
 obtainable optimal rate of convergence. Therefore, the parameter $a$ is often called { degree of ill-posedness} (c.f. \cite{Natterer84}).  
On the other hand, an increase of the value of $p$ or $s$ leads to a faster optimal rate. 
In other words values of a linear functional given by a smoother slope function or representer can be estimated 
faster, as expected. 
Moreover, in the cases \textit{(ppp)} and \textit{(epp)} 
the parametric rate $n^{-1}$ is obtained if and only if the functional  is \lq smoother\rq\ 
 than the degree of ill-posedness of $\op$ in  the sense that \textit{(i)} $s\geqslant a-1/2$ and \textit{(ii)} $s\geq a$. 
The situation is different  in the  cases \textit{(pep)} and \textit{(ppe)}, where the optimal rates are always logarithmic or parametric, respectively.
\hfill$\square$\end{rem}
\begin{rem}\label{rem:fin:upper:com:2} There is an interesting issue hidden in the parametrization that we have chosen.
 Consider a classical indirect regression model  given by the covariance operator $\op$ and Gaussian white noise $\dot W$,
 i.e., $g_n=\op \beta+n^{-1/2} \dot W$ (for details see e.g. \cite{HoffmannReiss04}). 
It is shown in \cite{JohannesKroll2008}  that, for example in case \textit{(ppp)}, the optimal rate of convergence over the classes 
$\cF_\bw^\br$ and $\cFF_\hw^\hr$ of any estimator of $ \ell(\sol)$ is of order $\max(n^{-(p+s)/(p+2a)},n^{-1})$. 
In contrast, 
Proposition \ref{lem:applic:1} states that the optimal rate in a functional linear model 
  is of order
 $\max(n^{-(p+s)/(p+a)},n^{-1})$.  
Thus, we see by comparing  the two rates  that the covariance operator $\op$ in a functional linear model has  
the {degree of ill-posedness} $a$ while the same operator has a {degree of ill-posedness} $(2a)$  in the indirect regression model. In other words, in a functional linear model we do not face the complexity of the
inversion of $\op$ but only of its square root $\op^{1/2}$. 
Similar remarks hold true for the other cases, however,  in case \textit{(pep)} the rate of convergence is the same as in an indirect regression model with Gaussian white noise 
(c.f. \cite{JohannesKroll2008}). This is due to the fact that  if
 $\Tw_{j}\asymp \exp(-r |j|^{2a})$ for some $r>0$, then the dependence of the rate of convergence  on the value 
$r$ is hidden in the constant.
\hfill$\square$\end{rem}
\subsection{The upper bound}
Proposition~\ref{res:gen:prop:cons}
 shows that the estimator $\widehat\ell_m$   defined in \eqref{gen:def:est}  is consistent for 
all slope functions  and  functionals belonging to $\cF_\bw$ and $\cFF_{1/\bw}$, respectively. The following theorem provides an upper bound if $\sol$ belongs to an ellipsoid $\cF_\bw^\br$.  
\begin{theo}\label{res:gen:upper} 
Let $\{(Y_i, X_i)\}_{i=1}^n$ be an i.i.d.\ sample of $(Y,X)$ satisfying
\eqref{intro:e1}.
Suppose that Assumptions \ref{ass:reg}, \ref{ass:MOM} and \ref{ass:Kappa}
  hold true and that  
\begin{equation}\label{gen:upper:varphi:cond}
\sup_{m\in\N}\Bigl\{\frac{\Tw_m}{\bw_m} m^{3} \Bigr\}<\infty.
\end{equation} Consider $\dstar$ 
as in
\eqref{eq:def:dstar} and the estimator  $\widehat\ell_m$  defined  with 
dimension $m:=\kstar$ given by \eqref{eq:def:kstar:astar}. 
 There exists a constant $\calc(d,\Tw,\bw)$ depending  on $d,\Tw$ and $\bw $  only	
such that $ \risk[\widehat\ell_{\kstar}, \cF_\bw^{\br},\cN_{\Tw}^{\Td} ]\leqslant
\calc(d,\Tw,\bw)\eta\{(\sigma^2+\br)\cdot \dstar+
\br\wnormV{\frep}{1/\bw}^2n^{-1}\}$ and therefore
 \[
\risk[\widehat\ell_{\kstar}, \cF_\bw^{\br},\cN_{\Tw}^{\Td} ]\leq  \calc(d,\Tw,\bw)  \eta(\sigma^2+\br+\br\wnormV{\frep}{1/\bw}^2)   
\cdot\inf_{\tilde\ell} 
\risk[\tilde\ell, \cF_\bw^{\br},\cN_{\Tw}^{\Td} ].
\]
\end{theo}
The rate 
$\delta^*:=(\dstar)_{n\geq 1}$ of the lower bound given  in Theorem \ref{res:lower} provides up to a constant also an upper bound of 
the estimator $\widehat{\ell}_{\kstar}$.  Thus, we have shown that  the rate 
$\delta^*$
 is optimal and hence  $\widehat{\ell}_{\kstar}$   is  minimax-optimal. 
 We observe that $\dstar\leq \hr\cdot \max_{1\leq j\leq \kstar}\Bigl({\Tw_j^{-1}\hw_j^{-1}}\Bigr)\astar=\hr\cdot\Dstar$ for all $\frep \in \cFF_\hw^\hr$
 and therefore we obtain the following result as a consequence of Theorem \ref{res:gen:upper} and we omit the proof.
\begin{coro}\label{coro:gen:upper} Under the conditions of Theorem \ref{res:gen:upper}
we have that \\$ \sup_{\frep\in\cFF_\hw^\hr}
\risk[\widehat\ell_{\kstar}, \cF_\bw^{\br},\cN_{\Tw}^{\Td} ]
\leqslant 
 \calc(d,\Tw,\bw)  \eta(\sigma^2+2\br)   
\cdot\hr\cdot\Dstar $ and therefore
\begin{gather*}
\sup_{\frep\in\cFF_\hw^\hr}
\risk[\widehat\ell_{\kstar}, \cF_\bw^{\br},\cN_{\Tw}^{\Td} ]
\leqslant  \calc(d,\Tw,\bw)  \eta(\sigma^2+2\br)   
\cdot \inf_{\tilde\ell}\sup_{\frep\in\cFF_\hw^\hr} 
\risk[\tilde\ell, \cF_\bw^{\br},\cN_{\Tw}^{\Td} ] .
\end{gather*}
\end{coro}  In the following we illustrate the previous results  by considering  the
typical choices of $\bw$ and $\Tw$
presented below Assumption \ref{ass:reg}.
\begin{prop}\label{lem:applic:2}
Let $\{(Y_i, X_i)\}_{i=1}^n$ be an i.i.d.\ sample of $(Y,X)$ satisfying
\eqref{intro:e1} and suppose that
 Assumption \ref{ass:MOM} holds true.
Let the estimator   $\widehat\ell_m$ be  defined  with  dimension $m:=\kstar$  as
  given in  Proposition \ref{lem:applic:1}.
Then ${\widehat\ell_{\kstar}}$ attains the optimal rates $\dstar$ and respectively
$\Dstar$ given in Proposition \ref{lem:applic:1}
 if we additionally assume $p+a\geq 3/2$ in the cases  
\textit{(ppp)} and \textit{(ppe)}. 
 \end{prop}
\begin{rem}\label{rem:comp:ch}
It is of interest to compare our results  with those of \cite{CaiHall2006} who 
consider only \textit{(i)} in the case \textit{(ppp)}. In their notations the
decay of the eigenvalues of $\op$ 
is assumed to be of order $(|j|^{-\otheralpha})$,
 i.e., $\otheralpha=2a$, with $\otheralpha>1$. 
Furthermore, they suppose  a decay of the coefficients of the  slope function and the representing sequence $[\ell]$  of order  $(|j|^{-\otherbeta})$, i.e., $\otherbeta=p+1/2$, with
$\otherbeta\geq \alpha +2$, and
$(|j|^{-\othergamma})$, i.e., $\othergamma=s$, with $\othergamma>1/2$ respectively. By using this parametrization we see that our results in the case \textit{(ppp)} imply the same
rate of convergence  as the one presented in \cite{CaiHall2006}. 
However, we shall stress that the condition  $\otherbeta\geq \otheralpha+2$ or equivalently  $p\geq 3/2 +2a$ is much stronger than the condition $p+a\geq3/2$ used in
Proposition~\ref{lem:applic:2}.
\hfill$\square$
\end{rem}
\subsection{Optimal point-wise and local average estimation}
\label{sec:illu}
%\setcounter{chapter}{3}
%\setcounter{equation}{0} %-1
%\setcounter{section}{0}
%\noindent {\bf \arabic{chapter}. Illustration.}
We continue the discussion of Section \ref{subs:pwlae}.
\paragraph{Point-wise estimation - continued} 
Recall that 
$\ell_{t_0}(\hsol_{\kstar})=\hsol_{\kstar}(t_0)$
 with $[\frep_{t_0}]^2_j=\bas_j^2(t_0)\asymp j^{-2s}$ and $s=0$.
By applying Proposition~\ref{lem:applic:2}, the estimator's maximal mean squared error over
the classes
 $\cF_\bw^{\br}$ and $\cN_{\Tw}^{\Td}$
 is uniformly bounded for $t_0\in[0,1]$ up to a constant by $\delta^*_n$, i.e., 
$
\sup_{\sol \in \cF_\bw^{\br}} \sup_{\op \in \cN_\Tw^{\Td}}\Ex|\widehat{\beta}_{\kstar}(t_0) -\sol(t_0)|^2
\leq C \dstar
 $ for some $C>0$. Moreover, due to Proposition~\ref{lem:applic:1},  $\delta^*_n$ is  the minimax-optimal rate of convergence.	
 This means in the three considered cases:
\begin{itemize}
 \item[\textit{(ppp)}] 
If  $p>1/2$, $a>1/2$ and  $p+a\geq3/2$, then $\kstar\asymp n^{1/(2p+2a)}$ and\\
$\dstar\sim n^{-(2p-1)/(2p+2a)}$.
 \item[\textit{(pep)}]  If $p>1/2$ and $a>0$, then $\kstar\asymp \log(n[\log(n)]^{-p/a})^{1/(2a)} $ and\\
$\dstar\asymp [\log(n)]^{-(2p-1)/2a}$.
 \item[\textit{(epp)}] If $p>0$ and $a>1/2$, then $\kstar\asymp \log(n[\log(n)]^{-a/p})^{1/(2p)}$
 and\\
$\dstar\asymp n^{-1}[\log(n)]^{(2a+1)/2p}$.
\end{itemize}
Let us  compare the optimal rate $n^{-(2p-1)/(2p+1)}$  in a direct regression model with a $p$-times differentiable  slope function $\sol$ with  our results in the cases \textit{(ppp)} and \textit{(pep)}. Obviously, the optimal rate in a functional linear model is never faster than the one of a direct regression model. Furthermore, they would  only coincide in the case \textit{(ppp)}  for $a=1/2$, however, this cannot hold true for any random function $X$ with finite second moment.\\
 It is interesting to note that by  slightly adapting the previously presented procedure we are able to  estimate the value of the $\sfs$-th derivative of $ \sol $ at $t_0$. 
 Given the exponential basis, which is linked 
 to the trigonometric basis 
 for $k \in \mathbb{Z}$ and $t \in [0, 1]$
 by the relation
   $\exp(2i \pi kt) = 2^{−1/2}
  (\bas_{2k}(t) + i \bas_{2k+1}(t))$ 
  with $i^2 = −1$. We recall that for $0\leq  \sfs < p$ the $\sfs$-th derivative $\sol^{(\sfs)}$
  of $\sol$ in a weak sense
  satisfies
 \[ \sol^{
  (\sfs)}(t_0) =
  \sum_{
  k\in\mathbb{Z}}
  (2i \pi k)^{\sfs}\exp(2i\pi kt_0)
 \Bigl(
 \int
 ^1
 _0
  \sol(u) 
 \exp(−2i\pi ku) 
 du
 \Bigr)
 .\]
  Given a dimension $m \geq 1$, we denote now by $[{\hop}]_{\um}$ the $(2m + 1) \times (2m + 1)$ matrix with
  generic elements $\HskalarV{\bas_j
  ,\hop \bas_k }$, $-m \leq j, k \leq m$ and by $[{\widehat{g}}]_{\um}$ the $(2m + 1)$ vector with elements
  $\HskalarV{\widehat{g}, \bas_j} $, $-m \leq j \leq m$. Furthermore, we define  for integer $\sfs$   the $(2m + 1)$ vector $[{\ell^{(\sfs)}_{t_0}}]_{\um}$ with elements
 $  
 [{\ell^{(\sfs)}_{t_0}}]_{j}:=(2i \pi j)^{\sfs}\exp(2i\pi j t_0)
 $, $-m \leq j \leq m$.
 In the following we shall assume that the point evaluation of the $\sfs$-th derivative  is well-defined on
 the set of slope parameters $\cF_\bw$ which is  implied by  
$\sum_{j\geq1}(j^{2\sfs}\bw_j^{-1})<\infty$,  since $|[\ell_{t_0}^{(\sfs)}]_j|^2\asymp j^{2\sfs}$.
 Obviously, this additional condition is automatically satisfied in case  \textit{(pep)} and requires the  assumption $\sfs<p-1/2$ in the cases  \textit{(ppp)} and \textit{(ppe)}. 
 We consider the estimator of $\sol
 ^{(\sfs)}(t_0)=\ell^{(\sfs)}_{t_0}(\sol)$
  given by
 \[
   \hsol^{(\sfs)}_m({t_0})=\begin{cases}                         [{\ell^{(\sfs)}_{t_0} }]_{\um}^t
[{\hop}]_{\um}^{-1}
 [{\hgf}]_{\um} &
 \text{if $[\hop]_{\um}$ is non-singular and }\mnormV{[\hop]^{-1}_{\um}}\leq  n,\\
 0,&\text{otherwise.}
                         \end{cases}
 \]
  The estimator $\hsol^{(\sfs)}_{\kstar}(t_0)$ 
 can be represented as $\frep^{(\sfs)}_{t_0}({\hsol_{\kstar}})$  with $[\frep_{t_0}^{(\sfs)}]_j^2\asymp j^{-2s}$ and $s=-q$. 
By applying Proposition~\ref{lem:applic:1} and~\ref{lem:applic:2}, the
 maximal mean squared error over
the classes
 $\cF_\bw^{\br}$ and $\cN_{\Tw}^{\Td}$ of the estimator 
$\hsol_{m}^{(\sfs)}(t_0)$ 
 is uniformly bounded for $t_0\in[0,1]$ up to a constant by the minimax rate $\delta^*_n$.
 This means in the three considered cases:
\begin{itemize}
 \item[\textit{(ppp)}] 
If  $p>1/2$, $a>1/2$ and  $p+a\geq3/2$, then $\kstar\asymp n^{1/(2p+2a)}$ and\\
$\dstar\sim n^{-(2p-2\sfs-1)/(2p+2a)}$.
 \item[\textit{(pep)}]  If $p>1/2$ and $a>0$, then $\kstar\asymp \log(n[\log(n)]^{-p/a})^{1/(2a)} $ and\\
$\dstar\asymp [\log(n)]^{-(2p-2\sfs-1)/2a}$.
 \item[\textit{(epp)}] If $p>0$ and $a>1/2$, then $\kstar\asymp \log(n[\log(n)]^{-a/p})^{1/(2p)}$
 and\\
$\dstar\asymp n^{-1}[\log(n)]^{(2a+2\sfs+1)/2p}$.
\end{itemize}
\hfill$\square$
\paragraph{Local average estimation - continued} Recall that $\hell^{b}_m= b^{-1}\int_0^b\widehat\sol_m(t)dt  $ with $[\frep_b]_j^2\asymp j^{-2s}$ and $s=1$.
Its maximal mean squared error  over $\cF_\bw^{\br}$ and $\cN_{\Tw}^{\Td}$ is bounded up to a constant by $\delta^*_n$, that is 
$
\sup_{\sol \in \cF_\bw^{\br},\op \in \cN_\Tw^{\Td}}\Ex|\int_0^b\widehat{\beta}_{\kstar}(t)dt -\int_0^b\sol(t)dt|^2
\leq C\dstar$ for some $C>0$ (Proposition~\ref{lem:applic:2}). Moreover, due to Proposition~\ref{lem:applic:1}, 
$\delta^*_n$
 is
 again  minimax-optimal. In the three cases  the order of $\delta^*_n$ is given as follows:
\begin{itemize}
 \item[\textit{(ppp)}] 
If  $p\geq 0$, $a>1/2$ and  $p+a>3/2$, then  $\kstar\asymp n^{1/(2p+2a)}$ and\\ $
\dstar\sim n^{-(2p+1)/(2p+2a)}$.
 \item[\textit{(pep)}]  If $p\geq 0$ and $a>0$, then  $\kstar\asymp \log(n[\log(n)]^{-p/a})^{1/(2a)} $ and\\ $
\dstar\asymp ~ [\log(n)]^{-(2p+1)/2a}$.
 \item[\textit{(epp)}] If $p>0$ and $a>1/2$, then $\kstar\asymp \log(n[\log(n)]^{-a/p})^{1/(2p)}$  and \\$
\dstar\asymp n^{-1}[\log(n)]^{(2a-1)/2p}$.
\end{itemize}
In contrast to a  direct regression model where the average value of the regression function can be estimated with parametric rate $n^{-1}$, we find that  in the considered cases the
optimal rate is always slower than $n^{-1}$.   Observe that in the  cases \textit{(ppp)} and \textit{(epp)} the rate could only be parametric  for $a\leq 1/2$, which again cannot hold true.   \hfill$\square$ 
\paragraph{Conclusion} In this paper we have presented a minimax optimal
plug-in estimation technique that is suited to deal in particular with
point-wise estimation and the estimation of local averages. Obviously, the data driven  choice of the dimension parameter $m$ is only one  amongst the many
interesting questions for further research and we are currently exploring
this topic. 
Another one might be the exploitation of local structures and sparse representations of the slope parameter $\sol$ by wavelet thresholding techniques.

\appendix
\section{Proofs}
%\setcounter{chapter}{4}
%\setcounter{equation}{0} %-1
%\setcounter{section}{0}
%\noindent {\bf \arabic{chapter}. Proofs.}
%\section{Proofs.}
\label{app:proofs}
We begin by defining and recalling notations to be used in the proofs of this section. Given $m\geq 1$, 
let us denote by
$\enormV{\cdot}$ the euclidean norm in $\R^m$,  by $\foum{\Diag(\bw)}_{\um}$
the $m$-dimensional diagonal matrix with entries $(\bw_1,\ldots,\bw_m)$ and
 by $\Id$ the $m$-dimensional identity matrix. Furthermore recall that  $\sol_{m}\in\sspace_m$  denotes the Galerkin solution of $\gf=\op\sol$ defined by \eqref{app:unknown:Galerkin:1:2}.
Let us  introduce the notations
\begin{multline*}
[\hop]_{\um}=\frac{1}{n}\sum_{i=1}^n[X_i]_{\um}[X_i]_{\um}^t,\quad
[\Xi]_{\um}:= [\op]_{\um}^{-1/2}[\hop]_{\um}[\op]_{\um}^{-1/2} - \Id,\\\quad 
[Z]_{\um}:=
[\widehat{g}]_{\um}- [\hop]_{\um} [\sol_m]_{\um}.
\end{multline*} 
 Moreover, we define the events 
 \begin{multline*}\label{app:l:upp:def:o}
\Omega:=\{ \mnormV{[\hop]^{-1}_{\um}}\leq \thresh n \},\quad  \Omega_{1/2}:= \{\mnormV{[\Xi]_{\um}}\leq 1/2\}\\
\Omega^c:=\{ \mnormV{[\hop]^{-1}_{\um}} > \thresh n \}\quad\mbox{ and }\quad  \Omega_{1/2}^c=\{\mnormV{[\Xi]_{\um}}> 1/2\}.
\end{multline*}
We shall prove in the end of this section two technical Lemmata (\ref{pr:gen:upper:l2}  and \ref{pr:gen:upper:l1}) which are used in the following proofs.   Furthermore, we will denote by $C$ universal numerical constants and by $C(\cdot)$ constants depending only on the arguments. In both cases, the values of the constants may change from line to line.
\subsection{Proof of the consistency result}\label{app:proofs:gen}
\begin{proof}[\noindent{\sc Proof of Proposition \ref{res:gen:prop:cons}.}]
 Let us define $\widetilde\ell_m:=\ell(\sol_m)\1_\Omega$. Then the proof is based on the  decomposition 
\begin{equation*}
\Ex|\widehat\ell_m - \ell(\sol)|^2\leqslant 2\{ \Ex|\widehat\ell_m- \widetilde\ell_m|^2 + \Ex|\widetilde\ell_m - \ell(\sol)|^2\}
\end{equation*}
where we will bound each term on the right hand side  separately. On the one hand we have
\begin{equation}\label{pr:gen:prop:cons:e2}
\Ex|\widetilde\ell_m - \ell(\sol)|^2\leqslant 2\{|\ell(\sol-\sol_m)|^2+ |\ell(\sol)|^2\,P(\Omega^c)\}.
\end{equation}
On the other hand, we conclude from the  identity $\fouv{\hgf}_{\um}-\foum{\hop}_{\um}\fouv{\sol_m}_{\um} =\fouv{Z}_{\um}$ that
\begin{equation*}
\Ex|\widehat\ell_m - \widetilde\ell_m|^2=
\Ex\absV{\fouv{\frep}^{t}_{\um}\,\{\foum{\op}_{\um}^{-1}+
(\foum{\hop}_{\um}^{-1} - \foum{\op}_{\um}^{-1})
\}\,\fouv{Z}_{\um}}^{2}\1_\Omega.
\end{equation*}
By using $ \mnormV{[\hop]_{\um}^{-1}}\1_{\Omega} \leqslant n$ and $\mnormV{\{\Id+\foum{\Xi}_{\um}\}^{-1}}\1_{\Omega_{1/2}}\leq 2 $
 it follows that
\begin{align*}
\Ex|\widehat\ell_m - \widetilde\ell_m|^2&\leqslant
 4\Bigl[ \Ex\absV{[\frep]^{t}_{\um}\,[\op]_{\um}^{-1}[Z]_{\um}}^2 
\\&+
\enormV{[\frep]^{t}_{\um}\,[\op]_{\um}^{-1/2}}^2  \Bigl\{4 (\Ex\mnormV{[\Xi]_{\um}}^4)^{1/2}(\Ex\enormV{[\op]_{\um}^{-1/2}[Z]_{\um}}^4)^{1/2}\\
&+  
n^2 \, \mnormV{[\op]_{\um}}^2 \,  (\Ex\mnormV{[\Xi]_{\um}}^8)^{1/4}(\Ex\enormV{[\op]_{\um}^{-1/2}[Z]_{\um}}^8)^{1/4}(P(\Omega_{1/2}^c))^{1/2}
\Bigr\}\Bigr].
\end{align*}
 We observe that $ \Ex\absV{[\frep]^{t}_{\um}\,[\op]_{\um}^{-1}[Z]_{\um}}^2\leqslant \enormV{[\frep]^{t}_{\um}\,[\op]_{\um}^{-1/2}}^2 \sup_{z\in\R^m: z^tz=1}\Ex\absV{z^t\,[\op]_{\um}^{-1/2}[Z]_{\um}}^2$ and
 that $\mnormV{[\op]_{\um}}$ is less or equal than the operator norm $\mnormV{\op}$ of $\op$, which equals its largest eigenvalue.
 Therefore,   by using  \eqref{pr:gen:upper:l2:e1} - \eqref{pr:gen:upper:l2:e3}   in Lemma \ref{pr:gen:upper:l2} with $k=12$ we have
\begin{multline}\label{pr:gen:prop:cons:e1}
\Ex|\widehat\ell_m - \widetilde\ell_m|^2\leqslant
\numc\,(\enormV{[\frep]^{t}_{\um}\,[\op]_{\um}^{-1/2}}^2/{n})\,\eta\,
\Bigl\{\HnormV{\op^{1/2}(\sol-\sol_m)}^{2}+\sigma^2 \Bigr\}\\
\cdot \Bigl\{ 1+ {m^3}/{n}+
 \eta^{-1/2}\mnormV{\op}^2{m^3}{n}\,(P(\Omega_{1/2}^c))^{1/2}
\Bigr\}.
\end{multline}
The combination of \eqref{pr:gen:prop:cons:e2} and \eqref{pr:gen:prop:cons:e1} leads to the estimate
\begin{multline}\label{pr:gen:prop:cons:e3:1}
\Ex|\widehat\ell_m - \ell(\sol)|^2
\leqslant \numc\,\Bigl\{ |\ell(\sol-\sol_m)|^2+ |\ell(\sol)|^2\,P(\Omega^c)
+   (\enormV{[\frep]^{t}_{\um}\,[\op]_{\um}^{-1/2}}^2/n)\, \eta  \\\cdot \{ \HnormV{\op^{1/2}(\sol-\sol_m)}^{2}+\sigma^{2}\}
\Bigl[1+ {m^3}/{n}+\eta^{-1/2}\mnormV{\op}^2\,{m^3}{n}(P(\Omega_{1/2}^c))^{1/2}
\Bigr)\Bigr]\Bigr\}. \end{multline}
{We observe that  the identity $[\hop]_{\um}= [\op]^{1/2}_{\um}\{\Id+[\Xi]_{\um}\}[\op]^{1/2}_{\um}$ implies by the usual Neumann series argument that if $\mnormV{[\Xi]_{\um}}\leqslant 1/2$  then  $\mnormV{[\hop]^{-1}_{\um}} \leqslant 2\mnormV{[\op]^{-1}_{\um}}$. Thereby,  $ n\geqslant 2 \mnormV{\foum{\op}^{-1}_{\um}}$ implies $\Omega^c \subset\Omega_{1/2}^c$.}
Furthermore, due to \eqref{pr:gen:upper:l2:e3}   in Lemma \ref{pr:gen:upper:l2} with $k=12$   we obtain by applying Markov's inequality that    $P(\Omega_{1/2}^c)\leqslant \numc \eta m^{24}n^{-12}$.
This leads to
\begin{multline}\label{pr:gen:prop:cons:e3:2}
\Ex|\widehat\ell_m - \ell(\sol)|^2
\leqslant \numc\,\Bigl\{ |\ell(\sol-\sol_m)|^2+ |\ell(\sol)|^2\,(m^{3}/n)^8n^{-4}\,\eta \\+ (\enormV{[\frep]^{t}_{\um}\,[\op]_{\um}^{-1/2}}^2/n)\, \eta \, \{\HnormV{\op^{1/2}(\sol-\sol_m)}^{2}+\sigma^{2}  \} 
\Bigl[ 1+ ({m^3}/{n})
+  
({m^{3}}/{n})^{5} 
\mnormV{\op}^2 \Bigr]\Bigr\}
 \end{multline}
 Furthermore, for each $\sol\in\cF_\bw
$, we have  $\wnormV{\sol-\sol_m}{\bw}=o(1)$ as $m\to\infty$ from condition \eqref{eq:suff:cond:Galerkin}, which implies 
$|\ell(\sol-\sol_m)|^2=o(1)$ and
 $\HnormV{\op^{1/2}(\sol-\sol_m)}=o(1)$ as $m\to\infty$ under Assumption  \ref{ass:reg}. 
Consequently,
the
conditions   $m^3=O(n)$
and $[\frep]_{\um}^t[\op]_{\um}^{-1}[\frep]_{\um}=o(n)$ as $n\to\infty$
  ensure the convergence to zero of  the  bound given  in \eqref{pr:gen:prop:cons:e3:2} as $n\to\infty$, which proves the result.
\end{proof}

\noindent{\sc Proof of Corollary \ref{res:gen:coro:cons}.}
First, we prove that $\op\in\cN_{\Tw}^{\Td}$ implies \eqref{eq:suff:cond:Galerkin}.
Let us denote by $\Pi_m$ and $\Pi_m^\perp$ the orthogonal projections on $\sspace_m$ and its orthogonal complement, respectively. On the one hand, we have   $\wnormV{\Pi_m^\perp  \sol}{\bw}=o(1)$  as $m\to\infty$
by Lebesgue's dominated convergence theorem. On the other hand, 
from the identity $[\Pi_m \sol-\sol_m]_{\um} = -[\op]_{\um}^{-1}[\op \Pi_m^\perp \sol]_{\um}$ 
we conclude  
  $\wnormV{\Pi_m \sol
  -\sol_m}{\bw}^2
\leq 2(1+d^2)\wnormV{\Pi_m^\perp \sol   }{\bw}^2$
for all $\op \in\cN_{\Tw}^{\Td}$, because the estimate \eqref{pr:gen:upper:l1:e3b} in Lemma \ref{pr:gen:upper:l1} implies
$\sup_{\wnormV{f}{\bw}=1} 
\enormV{
\foum{\Diag(\bw)}^{1/2}_{\um}
\foum{\op}_{\um}^{-1}
(\fouv{\op f}_{\um}-\foum{\op}_{\um}\fouv{f}_{\um})
}^2
\leqslant  2(1+\Td^2)$.
By combining the two results, we obtain the assertion.
It remains to show that \eqref{eq:bed:red} can be substituted by \eqref{eq:neu:3.3}. 
 Due to
 \eqref{pr:gen:upper:l1:e0} in Lemma \ref{pr:gen:upper:l1} the link condition  $\op \in\cN_{\Tw}^\Td$ implies ${\Tw_m}\mnormV{[\op]^{-1}_{\um}}\leqslant {4\Td^3}$. Since $1/\Tw_m=o(n)$  we conclude  that $\mnormV{[\op]^{-1}_{\um}}=o(n) $ as $n\to\infty$. 
 Furthermore, from   \eqref{pr:gen:upper:l1:e1} in Lemma \ref{pr:gen:upper:l1} we have
$
[\frep]^{t}_{\um}\,[\op]_{\um}^{-1}[\frep]^{t}_{\um}
\leq  \sum_{j=1}^{m}[\frep]_j^2\Tw_j^{-1}\mnormV{[\Diag(\Tw)]_{\um}^{1/2}[\op]_{\um}^{-1/2}}^2\leq 4d^3 \sum_{j=1}^{m}[\frep]_j^2\Tw_j^{-1}
$.
Thus $ \sum_{j=1}^{m}[\frep]_j^2\Tw_j^{-1} =o(n)$ implies $ [\frep]^{t}_{\um}\,[\op]_{\um}^{-1}[\frep]^{t}_{\um}=o(n) $
as $n\to\infty$,  which proves the result.
\hfill$\square$
\subsection{Proof of  the lower bound}\label{app:lower:proofs}
\begin{proof}[{\sc Proof of Theorem \ref{res:lower}.}]  
Consider $\epsilon$ and $X$ with  $\op\in\cN_{\Tw}^{\Td}$   such that Assumption  \ref{ass:MOM} is satisfied  and  
 $\epsilon|\HskalarV{X,f}\sim\cN(0,1)$  holds true for all $f\in\mmH$. Assume i.i.d.\ copies $\{(\epsilon_i,X_i)\}_{i=1}^n$   of $(\epsilon,X)$ and 
let $\sol_*\in\cF_{\bw}^\br$ with
$2\Td n \wnormV{\sol_*}{\Tw}^2\leq \sigma^2$, to be specified below.
Obviously, 
for
each $\theta \in \{-1,1\}$ the function $\sol_{\theta}:=\theta\, \sol_{*}$ belongs also to $\cF_\bw^\br$
and 
the random variables $\{(Y_i,X_i)\}_{i=1}^n$ with 
$Y_i:=\HskalarV{\sol_{\theta},X_i}+\sigma\epsilon_i$ form a sample of the model \eqref{intro:e1}. We denote its joint distribution by  $\PP_{\theta}$. 
We observe that  the conditional distribution of $Y_i$ given $X_i$  is Gaussian with mean  $
 \theta \HskalarV{\sol_*,X_i}
$  and variance $\sigma^2$. Thereby, it is easily seen that   the expectation of the log-likelihood function of ${\PP}_{1}$ with respect to  ${\PP}_{-1}$
 satisfies
\begin{equation*}
\Ex_{{\PP}_{-1}}[\log(d{\PP}_{1}/\d{\PP}_{-1})]= (2 n/\sigma^2) \, 
\HskalarV{\op\sol_*,\sol_*}= (2 n/\sigma^2)\HnormV{\op^{1/2}\sol_*}^2.
\end{equation*}In terms of  Kullback-Leibler divergence
this means  that the inequality  $\KL(\PP_{1},\PP_{-1})\\\leqslant 
(2d  n/\sigma^2)\wnormV{\sol_*}{\Tw}^2\leq 
1$ holds true,
 by  using the  inequality of \cite{Heinz51}, i.e.,  
for all $|s|\leq 1$, $f\in\mmH$ and $T\in\cN_{\Tw}^{\Td}$ we have 
$\HnormV{T^{s}f}^2\leq d^{2|s|}\wnormV{f}{\Tw^{2s}}^2$, together with the condition  $2\Td n \wnormV{\sol_*}{\Tw}^2\leq \sigma^2$.
Due to the bounded Kullback-Leibler divergence, Le Cam's general method (see \cite{Lecam1973}) and Pinsker's inequality  allow us 
to derive a lower bound. 
However, in this special setting a lower bound  can be obtained by the following elementary steps. We consider the  Hellinger affinity $\rho(\PP_{1},\PP_{-1})= \int \sqrt{\d\PP_{1}\d\PP_{-1}}$ and obtain for any estimator $\breve{\ell}$ and for all  $\frep\in \cFF_{1/\bw}$ that  
\begin{align}\nonumber
\rho(\PP_{1},\PP_{-1})&\leqslant 
\int \frac{
|
\breve{\ell}
-
\ell(\sol_{1})
|
}{2
|\ell(\sol_{*})|
} 
\sqrt{\d\PP_{1}\d\PP_{-1}} 
+
\int \frac{
|
\breve{\ell}
-
\ell(\sol_{-1})
|
}{2
|\ell(\sol_{*})|
} 
\sqrt{\d\PP_{1}\d\PP_{-1}} 
\\\label{pr:lower:e4}
&\leqslant 
\Bigl( 
\int  \frac{|
\breve{\ell}
-
\ell(\sol_{1})
|^2}
{4
|\ell(\sol_{*})|
^2}\d\PP_{1}
\Bigr)^{1/2}
+
\Bigl( 
\int  \frac{|
\breve{\ell}
-
\ell(\sol_{-1})
|^2}
{4
|\ell(\sol_{*})|^2}\d\PP_{-1}
\Bigr)^{1/2}.
\end{align}
By using the identity $\rho(\PP_{1},\PP_{-1})=1-\frac{1}{2}\H^2(\PP_{1},\PP_{-1})$  it follows  
from
 \eqref{pr:lower:e4}   that
\begin{equation}\label{pr:lower:e5}
\left\{\Ex_{{\PP_{1}}}|
\breve{\ell}
-
\ell(\sol_{1})
|^2
+ 
\Ex_{{\PP_{-1}}}
|\breve{\ell}
-
\ell(\sol_{-1})
|^2
\right\}\geqslant
\frac{1}{2}
|\ell(\sol_{*})|
^2
 \end{equation}
since the
 Hellinger distance $\H(\PP_{1},\PP_{-1})$ between $\PP_{1}$ and $\PP_{-1}$ satisfies
 $\H^2(\PP_{1},\PP_{-1}) \leqslant \KL(\PP_{1},\PP_{-1})\leq 1$.  
 From \eqref{pr:lower:e5} we conclude for each estimator $\breve{\ell}$ 
 that  
\begin{align}\notag
\sup_{\sol \in \cF_\bw^\br} &\Ex|\breve{\ell}-\ell(\sol)|^2 \geqslant \sup_{\theta\in \{-1,1\}} \Ex_{\PP_\theta}|\breve{\ell} -\ell(\sol_{\theta})|^2\\
&\geqslant \frac{1}{2}
\Bigl\{\Ex_{{\PP_{1}}}|
\breve{\ell}
-
\ell(\sol_{1})
|^2
+ 
\Ex_{{\PP_{-1}}}
|\breve{\ell}
-\ell(\sol_{-1})
|^2
\Bigr\}
\geqslant 
\frac{1}{4}|\ell(\sol_{*})|
^2 \label{pr:lower:e1}.
\end{align}
We will obtain the claimed result of the theorem by evaluating \eqref{pr:lower:e1} for two special choices of $\sol_*\in \cF_\bw^\br$
 with $2\Td n \wnormV{\sol_*}{\Tw}^2\leq \sigma^2$, which we will construct in the following.
Define $\zeta:=\min(\frac{\sigma^2}{2d},\br)$ and let $\kappa$ be given by \eqref{eq:def:kappa}.
 On the one hand, consider the slope function $\sol_{*}:= \sum_{j=1}^{\kstar}\fouv{\sol_{*}}_{j}\,\bas_{j}$, with coefficients $\fouv{\sol_{*}}_{j}:= {\fouv{\frep}_j}{\Tw_j^{-1}}
({{\zeta\kappa\astar
})^{1/2}({ \sum_{j=1}^{\kstar}{\fouv{\frep}_j^2}{\Tw_j^{-1}}}})^{-1/2}$.
 Since $\Tw/\bw$ is monotonically decreasing and by using the definition of $\kappa$ and $\zeta$  it follows that
 $\wnormV{\sol_{*}}{\bw}^2
 \leq {\zeta
 \kappa\astar\bw_{\kstar
 }{
 {\Tw_{\kstar}^{-1}}{}
 }
 }
 \leq \zeta\leq \br
 $
 and, hence $\sol_{*}\in\cF_\bw^\br$.
 Furthermore, we have
 that 
 $\label{pr:lower:e3}
  {2\,d\, n  }{}  \wnormV{\sol_*}{ \Tw}^2
 ={2d\zeta}{}{\kappa\astar}{n}\leq {2d\zeta}{}
 \leqslant \sigma^2.
 $
{Obviously, by
evaluating
 \eqref{pr:lower:e1} we  conclude  $\sup_{\sol \in \cF_\bw^\br} \Ex|\breve{\ell}-\ell(\sol)|^2\geq (\kappa/4)\zeta\astar\sum_{j=1}^{\kstar}{\fouv{\frep}_j^2}{\Tw_j^{-1}}$.} 
On the other hand, consider 
 $\sol_*:= \sum_{j>\kstar}\fouv{\sol_*}\bas_j$ with $ \fouv{\sol_*}_j:=({\kappa\zeta})^{1/2}(\bw_j^{2}{\sum_{j>\kstar}\fouv{\frep}_j^2\bw_j^{-1}})^{-1/2}
\fouv{\frep}_j$  we conclude from $\kappa\leq 1$ and
$
\wnormV{\sol_*}{\bw}^2=\sum_{j>\kstar}\fouv{\sol_*}_j^2\bw_j=\zeta\kappa\leq \rho
$ that $\sol_*$  belongs to $\cF_\bw^\br$.
Moreover, we have
$\label{eq:einschub:1}
{2nd}\wnormV{\sol_*}{\Tw}^2
\leq {2nd\zeta}{}{\kappa\Tw_{\kstar}}{{\bw_{\kstar}}^{-1}}\leq{2d\zeta}{}\leq\sigma^2.
$
By evaluating \eqref{pr:lower:e1} we obtain
$
\sup_{\sol \in \cF_\bw^\br} \Ex|\breve{\ell}-\ell(\sol)|^2 
\geqslant ({\kappa}/{4})\zeta
\sum_{j>\kstar}\fouv{\frep}_j^2\bw_j^{-1}.
$
Combining the two lower bounds, which hold true for arbitrary $\op\in\cN_{\Tw}^{\Td}$, 
we obtain 
\[\inf_{\tilde\ell}\inf_{\op\in\cN_{\Tw}^{\Td}}\sup_{\sol \in \cF_\bw^\br} \Ex|\breve{\ell}-\ell(\sol)|^2\geq \frac{\kappa}{4}\zeta
\max\Bigl\{
\astar\sum_{j=1}^{\kstar}{\fouv{\frep}_j^2}{\Tw_j^{-1}}, \sum_{j>\kstar}\fouv{\frep}_j^2\bw_j^{-1}  \Bigl\},\]
which implies the result of the theorem.
 \end{proof}
\noindent{\sc Proof of Proposition \ref{lem:applic:1}.}
We start our proof with the observation that under the conditions on $p$, $a$ and $s$ given in the proposition the sequences $\bw$, $\Tw$ and $\hw$ satisfy Assumption \ref{ass:reg} and Assumption \ref{ass:Kappa}.\\
Proof of \textit{(ppp)}.
From the definition of $\kstar$ in \eqref{eq:def:kstar:astar} it follows that $\kstar\sim n^{1/(2a+2p)}$.
Consider case \textit{(i)}. The condition $s-a<1/2$ implies
$n^{-1}\sum_{j=1}^\kstar |j|^{2a-2s}\sim n^{-1}(\kstar)^{2a-2s+1}\sim n^{-(2p+2s-1)/(2p+2a)}$
 and
moreover we have
$\sum_{j>\kstar}|j|^{-2p-2s}\sim n^{-(2p+2s-1)/(2p+2a)}$ since $p+s>1/2$.
If $s-a=1/2$, then
 $n^{-1}\sum_{j=1}^\kstar|j|^{2a-2s}\sim n^{-1}\log(n^{1/(2p+2a)})$ and 
$\sum_{j>\kstar}|j|^{-2p-2s}\sim n^{-1}$.
In the case $s-a>1/2$ it follows that $\sum_{j=1}^\kstar|j|^{2a-2s}$
is bounded. Moreover,  there exists a constant $c>0$ such that
 $\sum_{j>\kstar}|j|^{-2p-2s}\leq c\cdot n^{-1}$, or   $\sum_{j>\kstar}|j|^{-2p-2s}\lesssim n^{-1}$ for short,
and hence 
$\dstar\sim n^{-1}$.
To prove  \textit{(ii)}\ 
we make use of Corollary \ref{coro:lower:alt}. We observe that if $s-a\geq0$ the sequence $\hw\Tw$ is bounded from below,
and hence $\dstar \sim n^{-1}$. On the other hand, the condition $s-a<0$ implies 
$\dstar\sim n^{- (p+s)/(p+a)}$.\\
Proof of \textit{(pep)}. If  $\Tw$ is exponentially decreasing, then $\kstar$ satisfies
$\exp(-(\kstar)^{2a})\sim n^{-1}(\kstar)^{2p}$ or equivalently 
$\kstar\sim \log(n[\log(n)]^{-p/a})^{1/(2a)}$. To prove \textit{(i)}, we calculate 
$\sum_{j>\kstar}|j|^{-2p-2s}\sim [\log(n)]^{(-2p-2s+1)/(2a)}$ and 
$n^{-1}\sum_{j=1}^\kstar\exp(|j|^{2a})|j|^{-2s}\lesssim n^{-1}\exp(\kstar^{2a})\sim 
[\log (n)]^{(-2p-2s+1)/(2a)}$. In case \textit{(ii)}\ we immediately obtain
 $\dstar\sim [\log(n)]^{-(p+s)/a}$.\\
Proof of \textit{(epp)}. Only  $\bw$ is an exponential sequence and hence we have $\kstar\sim n^{-1}\exp((\kstar)^{2p})$ or equivalently 
$\kstar\sim \log(n[\log(n)^{-a/p}])^{1/(2p)}$. Consider case \textit{(i)}. 
If $s-a<1/2$, then
$n^{-1}\sum_{j=1}^\kstar |j|^{2a-2s}\sim n^{-1}[\log(n)]^{(2a-2s+1)/(2p)}$.
If $s-a=1/2$, then
$n^{-1}\sum_{j=1}^\kstar |j|^{2a-2s}\sim n^{-1}\log(\log(n))$.
On the other hand, the condition $s-a>1/2$ implies that 
$\sum_{j=1}^\kstar |j|^{2a-2s} $ is bounded and thus, we obtain the parametric rate $n^{-1}$. By using
$\sum_{j>\kstar}|j|^{-2s}\exp(-|j|^{2p})\lesssim \exp(-(\kstar)^{2p})(\kstar)^{2s-2p+1} \\\sim
n^{-1}[\log(n)]^{(-2a-2s-2p+1)/(2p)}$ it is easily seen that this sum is dominated by
$n^{-1}\sum_{j=1}^\kstar |j|^{2a-2s}$. 
In case \textit{(ii)}\ if $s-a\geq 0$
 then the sequence $\hw\Tw$ is bounded from below as mentioned 
above and thus,
 $\dstar\sim n^{-1}$. 
If $s-a<0$ then $\dstar\sim n^{-1}[\log(n)]^{(a-s)/p}$.\\
Proof of \textit{(ppe)}. As in case of \textit{(ppp)}\ both sequences $\bw$ and $\Tw$ are polynomial and thus $\kstar \sim n^{1/(2p+2a)}$.
Consider case \textit{(i)}\ where the coefficients of $\rep$ decrease exponentially. Obviously the sum
$\sum_{j=1}^\kstar|j|^{2a}\exp(-|j|^{2s})$ is bounded and moreover, $\sum_{j>\kstar}|j|^{-2p}\exp(-|j|^{2s})\lesssim
\exp(-|\kstar|^{2s})|\kstar|^{-2s-2p+1}\lesssim n^{-1}$. Consequently, we have $\dstar\sim n^{-1}$. Also in case \textit{(ii)} it
 obviously holds $\dstar\sim n^{-1}$, which completes the proof.\hfill$\square$
\subsection{Proof of the upper bound}
The following technical lemma is used in the proof of Theorem  \ref{res:gen:upper}.
\begin{lem}\label{lem:zusatz:Wkeit} If the assumptions of Theorem \ref{res:gen:upper} hold true, then  we have  $P(\Omega^c)\leq C(\bw,\Td,\Tw)\, \eta\, n^{-1}$.
\end{lem}
\begin{proof}[{\sc Proof.}]
Our proof starts with the observations that $\kappa\bw_{\kstar}\leq n\Tw_{\kstar}$ for all $n\geq 1$ by exploiting Assumption \ref{ass:Kappa} and that  $2\mnormV{[\op]^{-1}_{\um}}\leqslant 8\Td^3\Tw_{m}^{-1}$ for all  $\op \in\cN_{\Tw}^\Td$ and $m\geq1$ due to
\eqref{pr:gen:upper:l1:e0} in Lemma \ref{pr:gen:upper:l1}. Combining both estimates with $ \bw_{\kstar}^{-1}=o(1)$ as $n\to\infty$ we conclude $2\mnormV{[\op]^{-1}_{\ukstar}}=o(n)$. Therefore, there exists an integer $n_0:=n_0(\bw,\Td,\Tw)$ such that for all $n\geq n_0$ we have $2\mnormV{[\op]^{-1}_{\ukstar}}\leq n^{-1}$, and particularly  $\Omega_{1/2}\subset \Omega$ by applying the usual Neumann series argument.
We distinguish in the following the cases $n<n_0$ and $n\geq n_0$. Consider first  $n\geq n_0$,  from  Markov's inequality together with    \eqref{pr:gen:upper:l2:e3}   in Lemma
\ref{pr:gen:upper:l2} we obtain $P(\Omega^c)\leq P(\Omega_{1/2}^c)\leqslant \numc \eta (\kstar)^{6}n^{-3}$. Taking into account the condition \eqref{gen:upper:varphi:cond}, that is  
$D:= D(\bw,\Tw):= \sup_{m\in\N}\Bigl\{\frac{\Tw_m}{\bw_m} m^{3} \Bigr\}<\infty$,
 we
have $(\kstar)^{6}n^{-3}\leq ((\kstar)^3\Tw_{\kstar}\bw_{\kstar}^{-1}\kappa^{-1})^{2}n^{-1}\leq D^2\kappa^{-2}n^{-1}$, and hence $P(\Omega^c)\leq \numc \eta
D^2\kappa^{-2}n^{-1}$ for all $n\geq n_0$. On the other hand, if $n<n_0$ then trivially $P(\Omega^c)\leq n^{-1}n_0$. Since $n_0$, $D$ and $\kappa$ depend  on
$\bw,\Td$ and $\Tw$ only we obtain the result by combining both cases, which completes the proof.
\end{proof}
\begin{proof}[\noindent{\sc Proof of Theorem \ref{res:gen:upper}.}] 
Consider again the bound \eqref{pr:gen:prop:cons:e3:1}. By 
using that $\mnormV{\op} \leq d$ for $ \op\in\cN_{\Tw}^{\Td}$, $ (\kstar)^3n^{-1}\leq D\kappa^{-1}$ with $D= \sup_{m\in\N}\Bigl\{\frac{\Tw_m}{\bw_m} m^{3} \Bigr\}<\infty$
and recalling that  $P(\Omega_{1/2}^c)\leq C\eta ({\kstar})^{24}n^{-12}$
 we obtain
\begin{multline*}
\Ex|\widehat\ell_{\kstar} - \ell(\sol)|^2\leqslant \numc\,\Bigl\{ |\ell(\sol-\sol_{\kstar})|^2+ |\ell(\sol)|^2P( \Omega^c)\\
 +(\enormV{[\frep]^{t}_{\ukstar}\,[\op]_{\ukstar}^{-1/2}}^2/n)\, \eta \, \{\HnormV{\op^{1/2}(\sol-\sol_{\kstar})}^{2}+\sigma^{2}\}\,\Bigl[1+ D\kappa^{-1}+d^2\,
 {D^5}{\kappa^{-5}}\Bigr]\Bigr\}.
\end{multline*}
 From \eqref{pr:gen:upper:l1:e1}, \eqref{pr:gen:upper:l1:e3} and  \eqref{pr:gen:upper:l1:e4} in Lemma \ref{pr:gen:upper:l1} we conclude
$(\enormV{[\frep]^{t}_{\ukstar}\,[\op]_{\ukstar}^{-1/2}}^2/n)\leq  4d^3\dstar$, furthermore
  $\HnormV{\op^{1/2}(\sol-\sol_{\kstar})}^{2}\leqslant 10 \Td^5¸\wnormV{\sol}{\bw}^2$ and $|\ell(\sol-\sol_{\kstar})|^2\leqslant
16\wnormV{\sol}{\bw}^2 d^4\dstar$,  respectively.
Therefore we conclude for all $\sol\in\cF_\bw^\br$
\begin{equation*}
\Ex|\widehat\ell_{\kstar} - \ell(\sol)|^2\leqslant \numc\,\Bigl\{\br d^4\dstar+ \br\wnormV{\frep}{1/\bw}^2P( \Omega^c) 
+ d^3\, \eta \,\dstar \{\br d^5+\sigma^{2}\}\, \Bigl[1+ D\kappa^{-1}+ d^2 {D^5}{\kappa^{-5}}\Bigr]\Bigr\}.
\end{equation*}
Observe that $D$ and $\kappa$  depend on $\Tw$ and $\bw$ only and thus
by applying  Lemma \ref{lem:zusatz:Wkeit} we obtain
\begin{equation*}
\Ex|\widehat\ell_{\kstar} - \ell(\sol)|^2
\leqslant 
C(d,\Tw,\bw)\,\Bigl\{
 \br\dstar+ \br\,\eta\,\wnormV{\frep}{1/\bw}^2\,
 n^{-1}
+  
 \eta \, \{
\br
+\sigma^{2}\} \dstar
\Bigr\},
\end{equation*}
which completes the proof.\end{proof}
\noindent{\sc Proof of Proposition \ref{lem:applic:2}.}
Under the stated conditions  it is easy to verify that the assumptions of Theorem \ref{res:gen:upper} are satisfied.
The result follows  by applying  Theorem \ref{res:gen:upper}
 and 
Corollary \ref{coro:gen:upper} and we omit the details.\hfill $\square$ \par
\section{Technical assertions}
The following two lemmata gather technical results used in the proof of Proposition \ref{res:gen:prop:cons} and  Theorem \ref{res:gen:upper}. 
\begin{lem}\label{pr:gen:upper:l2}
Under Assumption \ref{ass:MOM} there exists a constant $C(k)>0$  such that
\begin{gather}\label{pr:gen:upper:l2:e1}
\sup_{z\in\R^m: z^tz=1} \Ex\Bigl|z^t [\op]_{\um}^{-1/2}[Z]_{\um}\Bigr|^{2k} \leqslant  
C(k) \, n^{-k}\,\Bigl(\sigma^2+\HnormV{\op^{1/2} (\sol-\sol_m)}^2
\Bigr)^k\, \eta,\\ \label{pr:gen:upper:l2:e2}
\Ex\enormV{[\op]_{\um}^{-1/2} [Z]_{\um}}^{2k}\leq C(k)\, \frac{m^{k}}{n^k}\, \,\Bigl(\sigma^{2}+\HnormV{\op^{1/2} (\sol-\sol_m)}^{2}\Bigr)^{k}\, \eta,\\ \label{pr:gen:upper:l2:e3}
\Ex\mnormV{[\Xi]_{\um}}^{2k} \leq C(k)\cdot\eta\cdot \frac{m^{2k}}{n^k}.
\end{gather}
\end{lem}
\begin{proof}[\noindent{\sc Proof.}]
Let us begin by deriving elementary bounds due to Assumption \ref{ass:MOM}.
For $m\geq 1$ define $U:=\sigma\epsilon+\HskalarV{\sol-\sol_m,X}$, where $\sigma^2_{U}=\Var(U)=\sigma^2+\HnormV{\op^{1/2} (\sol-\sol_m)}^2$. It is easily seen that for $k$ as given in  Assumption \ref{ass:MOM} and all $m\geq 1$ we have
\begin{multline}\label{eq:tec:lem:1} \Ex|U|^{4k}\leq C(k)\sigma_{U}^{4k}\eta,\qquad
\max_{1\leq j\leq m} \Ex|  ([\op]_{\um}^{-1/2}[X]_{\um})_j   |^{4k}\leq \eta \\\text{and }
 \sup_{z\in\R^m: z^tz=1} \Ex\Bigl|z^t [\op]_{\um}^{-1/2}[X]_{\um}\Bigr|^{4k} \leqslant \eta.
\end{multline}
 Let $z\in\R^m$ satisfy $z^tz=1$ and define  $U_i:=\sigma\epsilon_i+\HskalarV{\sol-\sol_m,X_i}$ then
$z^t [\op]_{\um}^{-1/2}[Z]_{\um}= \frac{1}{n}\sum_{i=1}^n U_iz^t[\op]_{\um}^{-1/2} [X_i]_{\um}$.
 Since $ \Ex \HskalarV{\sol-\sol_m,X}[X]_{\um}= [\op(\sol - \sol_m)]_{\um}= [g]_{\um}-[\op]_{\um}[\sol_m]_{\um}=0$, it follows that the random variables 
$U_iz^t[\op]_{\um}^{-1/2} [X_i]_{\um}$,  $i=1,\dots,n,$ are i.i.d. with mean zero. From Theorem 2.10 in \cite{Petrov1995} we conclude 
$\Ex |z^t [\op]_{\um}^{-1/2}[Z]_{\um}|^{2k} \leqslant C(k) n^{-k}\Ex |Uz^t[\op]_{\um}^{-1/2} [X]_{\um}|^{2k}$ for some  constant $C(k)>0$. 
Then  we claim that \eqref{pr:gen:upper:l2:e1} follows  from  the Cauchy-Schwarz inequality together with the bounds given in \eqref{eq:tec:lem:1}.
To deduce  \eqref{pr:gen:upper:l2:e2} from \eqref{pr:gen:upper:l2:e1} we use that
\begin{equation*}\Ex\enormV{ [\op]_{\um}^{-1/2}[Z]_{\um}}^{2k}
\leqslant  m^{k} \max_{1\leq j\leq m}\Ex\Bigl| ([\op]_{\um}^{-1/2}[Z]_{\um})_j\Bigr|^{2k} \leqslant m^k \sup_{z\in\R^m:z^tz=1} \Ex\Bigl| z^t [\op]_{\um}^{-1/2}[Z]_{\um}\Bigr|^{2k}.\end{equation*}
 Proof of \eqref{pr:gen:upper:l2:e3}. From the identity $n([\Xi]_{\um})_{j,l}= \sum_{i=1}^n \{ ([\op]_{\um}^{-1/2} [ X_i]_{\um})_j([\op]_{\um}^{-1/2}[
 X_i]_{\um})_{l}-\delta_{jl}\}$ with $\delta_{jl}= 1$ if $j=l$ and zero otherwise, we conclude  
$\Ex ([\Xi]_{\um})_{j,l}^{2k}\leq C(k)n^{-k}\eta$ by using  Theorem 2.10 in \cite{Petrov1995} and the second bound given in \eqref{eq:tec:lem:1}. 
Then using the elementary inequality
$\Ex\mnormV{ [\Xi]_{\um}}^{2k}\leq m^{2k} \max_{1\leq j,l\leq m} \Ex ([\Xi]_{\um})_{j,l}^{2k}$ 
 implies \eqref{pr:gen:upper:l2:e3}, which completes the proof.
\end{proof}
The next Lemma is partially shown in \cite{CardotJohannes2008}.
\begin{lem}\label{pr:gen:upper:l1}
Suppose the sequences $\bw$ and $\Tw$ satisfy Assumption \ref{ass:reg}.  Then we have for $T\in\cN_{\Tw}^\Td$
\begin{gather}\label{pr:gen:upper:l1:e0}
\sup_{m\geq 1}\Bigl\{ \Tw_m \mnormV{[T]_{\um}^{-1}}\Bigr\}\leqslant \{2\Td^2(2\Td^4+3)\}^{1/2}\leqslant 4\Td^3,\\
\label{pr:gen:upper:l1:e1}
\sup_{m\geq 1}\mnormV{[T]_{\um}^{-1/2}[\Diag(\Tw)]^{1/2}_{\um}}^2\leqslant \{2\Td^2(2\Td^4+3)\}^{1/2}\leqslant 4\Td^3.
\end{gather}
If in addition  $\sol_m$  denotes a Galerkin solution of $g=T\sol$ then
\begin{gather} \label{pr:gen:upper:l1:e3a}
\sup_{m\geq1}\Bigl\{\sup_{\HnormV{\beta}=1}\HnormV{\Pi_m\sol-\sol_m}^2\Bigr\}\leqslant 2(1+\Td^2), \\\label{pr:gen:upper:l1:e3b}
\sup_{m\geq1}\Bigl\{\sup_{\wnormV{\beta}{\bw}=1}\wnormV{\Pi_m\sol-\sol_m}{\bw}^2\Bigr\}\leqslant 2(1+\Td^2),
\end{gather}
and if  $\sol\in \cF_\bw^\br$ is  additionally satisfied then 
\begin{gather} 
\label{pr:gen:upper:l1:e3}
\sup_{m\geq1}\{\Tw_m^{-1}\bw_m\,\HnormV{\op^{1/2}(\sol-\sol_m)}^2\}\leqslant 10 \Td^5 \br.
\end{gather}
Furthermore, for all $m\geq1$ and all $\frep\in \cFF_{1/\bw}$  we have
\begin{gather} \label{pr:gen:upper:l1:e4}
|\frep({\sol-\sol_m})|^2\leqslant
2\wnormV{\sol}{\bw}^2\{\sum_{j>m}[\frep]_j^2\bw_j^{-1}+ 2(1+d^4)\frac{\Tw_m}{\bw_m}\sum_{j=1}^m[\frep]_j^2\Tw_j^{-1}\}.
\end{gather}
\end{lem}
\begin{proof}[{\sc Proof.}] The estimates \eqref{pr:gen:upper:l1:e0} - \eqref{pr:gen:upper:l1:e1} are given in Lemma A.3 in \cite{CardotJohannes2008}. Furthermore,  from (A.19) and (A.20) in Lemma A.3 in \cite{CardotJohannes2008} follow \eqref{pr:gen:upper:l1:e3a} and \eqref{pr:gen:upper:l1:e3b}.
We start our proof of \eqref{pr:gen:upper:l1:e3} with the observation that   the link condition $T\in\cN_{\Tw}^\Td$ implies that $T$ is strictly positive and that for all $|s|\leqslant 1$  by using the inequality of \cite{Heinz51}
\begin{equation}\label{pr:gen:upper:l1:Heinz}
\Td^{-2|s|} \wnormV{f}{\Tw^{2s}}^2 \leqslant \HnormV{T^sf}^2\leqslant \Td^{2|s|} \wnormV{f}{\Tw^{2s}}^2 .
\end{equation}
Thus, by using successively  the first inequality of \eqref{pr:gen:upper:l1:Heinz},  the Galerkin condition \eqref{app:unknown:Galerkin:1:2} and  the second inequality of \eqref{pr:gen:upper:l1:Heinz}, we obtain
\begin{equation}\label{pr:gen:upper:l1:e1:0}
\wnormV{\sol-\sol_m}{\Tw^2}^2\leqslant \Td^2 \HnormV{T(\sol-\sol_m)}^2\leqslant \Td^2 \HnormV{T(\sol-\Pi_m\sol)}^2\leqslant \Td^4 \wnormV{\sol-\Pi_m\sol}{\Tw^2}^2
\end{equation}
Since $\sol\in \cF_\bw^\br$ and $\bw^{-1}\Tw^2$ is monotonically decreasing we have $\wnormV{\sol-\Pi_m\sol}{\Tw^2}^2\leqslant  \bw_m^{-1}\Tw_m^2
\wnormV{\sol}{\bw}^2$, which together with \eqref{pr:gen:upper:l1:e1:0} implies $\wnormV{\sol-\sol_m}{\Tw^2}^2\leqslant  \Td^4 \bw_m^{-1}\Tw_m^2 \wnormV{\sol}{\bw}^2$ 
and hence,
\begin{equation}\label{pr:gen:upper:l1:e3:1}
\wnormV{\Pi_m\sol-\sol_m}{\Tw^2}^2\leqslant 2\{ \wnormV{\sol-\sol_m}{\Tw^2}^2 + \wnormV{\sol-\Pi_m\sol}{\Tw^2}^2\}\leqslant
2(1+\Td^4) \bw_m^{-1}\Tw_m^2 \wnormV{\sol}{\bw}^2.
\end{equation}
The last  estimate and  the second inequality of \eqref{pr:gen:upper:l1:Heinz} imply further  $\HnormV{\op^{1/2}(\Pi_m\sol-\sol_m)}^2\leq d\wnormV{\Pi_m\sol-\sol_m}{\Tw}^2\leq \Td
\Tw_m^{-1}\wnormV{\Pi_m\sol-\sol_m}{\Tw^2}^2\leq 2\Td(1+\Td^4) \bw_m^{-1}\Tw_m \wnormV{\sol}{\bw}^2$ because $\Tw$ is monotonically non increasing. Taking into account
$\wnormV{\sol-\Pi_m\sol}{\Tw}^2\leqslant  \bw_m^{-1}\Tw_m\wnormV{\sol}{\bw}^2$ we obtain \eqref{pr:gen:upper:l1:e3}. Finally, by applying the Cauchy-Schwarz inequality we have on
the one hand $|\frep({\sol-\Pi_m\sol})|^2\leqslant \wnormV{\sol}{\bw}^2\sum_{j>m}[\frep]_j^2\bw_j^{-1}$ and by using \eqref{pr:gen:upper:l1:e3:1} it follows on the other hand
$|\frep(\Pi_m\sol-\sol_m)|^2\leqslant \wnormV{\Pi_m\sol-\sol_m}{\Tw}^2\sum_{j=1}^m[\frep]_j^2\Tw_j^{-1}\leq  2(1+\Td^4)\wnormV{\sol}{\bw}^2 \bw_m^{-1}\Tw_m\sum_{j=1}^m[\frep]_j^2\Tw_j^{-1}$.
Combining both estimates implies now \eqref{pr:gen:upper:l1:e4}, which completes the
proof.\end{proof}

\section*{Acknowledgements}
This work was supported by the IAP research network no.\ P6/03 of the
Belgian Government (Belgian Science Policy) and by the ``Fonds
Sp\'eciaux de Recherche'' from the Universit\'e catholique de Louvain.

\bibliography{JJRS}
\end{document}